\documentclass{amsart}
\usepackage{breqn}
\usepackage{schock}

\title{The $W(E_6)$-invariant birational geometry of the moduli space of marked cubic surfaces}

\author{Nolan Schock}
\email{nschoc2@uic.edu}
\address{Department of Mathematics, University of Illinois at Chicago, Chicago
IL, 60607, USA}

\begin{document}

\begin{abstract}
    The moduli space $Y = Y(E_6)$ of marked cubic surfaces is one of the most classical moduli spaces in algebraic
    geometry, dating back to the nineteenth century work of Cayley and Salmon. Modern interest in $Y$ was restored in
    the 1980s by Naruki's explicit construction of a $W(E_6)$-equivariant smooth projective compactification $\oY$ of
    $Y$ \cite{narukiCrossRatioVariety1982}, and in the 2000s by Hacking, Keel, and Tevelev's construction of the KSBA
    stable pair compactification $\wY$ of $Y$ as a natural sequence of blowups of $\oY$
    \cite{hackingStablePairTropical2009}. We describe generators for the cones of $W(E_6)$-invariant effective divisors
    and curves of both $\oY$ and $\wY$. For Naruki's compactification $\oY$, we further obtain a complete stable base
    locus decomposition of the $W(E_6)$-invariant effective cone, and as a consequence find several new
    $W(E_6)$-equivariant birational models of $\oY$. Furthermore, we fully describe the log minimal model program for
    the KSBA compactification $\wY$, with respect to the divisor $K_{\wY} + cB + dE$, where $B$ is the boundary and $E$
    is the sum of the divisors parameterizing marked cubic surfaces with Eckardt points.
\end{abstract}
\maketitle

\section{Introduction}

The moduli space $Y = Y(E_6)$ of marked cubic surfaces is one of the most classical moduli spaces in algebraic geometry,
essentially dating back to the original works of Cayley and Salmon studying the 27 lines on a cubic surface
\cite{cayleyTripleTangentPlanes2009}. There is a natural action of the Weyl group $W(E_6)$ on $Y$, permuting the 27
lines, and using ideas dating back to Cayley and Coble \cite{cobleAlgebraicGeometryTheta1968}, in the 1980s Naruki
explicitly constructed a remarkable $W(E_6)$-equivariant smooth projective compactification $\oY = \oY(E_6)$ of $Y$,
with simple normal crossings boundary \cite{narukiCrossRatioVariety1982}. Naruki's compactification and related spaces
have since been intensively studied from a number of different perspectives, see
\cite{allcockComplexHyperbolicGeometry2002, allcockCubicSurfacesBorcherds2000,
    casalaina-martinNonisomorphicSmoothCompactifications2022, colomboChowGroupModuli2004, colomboPezzoModuliRoot2009,
cuetoAnticanonicalTropicalCubic2019, dolgachevComplexBallUniformization2005, freitagGradedAlgebraRelated2002,
gallardoGeometricInterpretationToroidal2021, hackingStablePairTropical2009, renTropicalizationPezzoSurfaces2016,
schockModuliWeightedStable2023, vangeemenLinearSystemNaruki2002, yoshidaE_6EquivariantProjective2000}.

The goal of the present article is to study the $W(E_6)$-invariant birational geometry of $\oY$, together with a related
space $\wY = \wY(E_6)$, a blowup of $\oY$ along intersections of certain boundary divisors. The space $\wY$ is of
significant interest, as it is the KSBA stable pair compactification (cf. \cite{kollarFamiliesVarietiesGeneral2023}) of
the moduli space of pairs $(S,B)$, where $S$ is a smooth cubic surface without Eckardt points, and $B$ is the sum of its
finitely many lines \cite[Theorem 10.31]{hackingStablePairTropical2009}. (More generally, if $c \in (1/9,1]$, so that
$K_S+cB$ is ample, then the KSBA compactification of the moduli space of weighted marked cubic surfaces $(S,cB)$ is
either $\oY$, for $c \in (1/9,1/4]$, or $\wY$, for $c \in (1/4,1]$ \cite[Theorem 1.2]{schockModuliWeightedStable2023}.)

To state our main results, let us first recall Naruki's compactification $\oY$ in more detail. The moduli space $\oY$ is
a 4-dimensional smooth projective variety with 76 irreducible boundary divisors, in bijection with certain root
subsystems of $E_6$---36 $A_1$ divisors $\cong \oM_{0,6}$ (parameterizing marked cubic surfaces with $A_1$
singularities) and 40 $A_2^3=A_2 \times A_2 \times A_2$ divisors $\cong (\bP^1)^3$ (parameterizing marked cubic surfaces
which are the union of three planes, e.g., $xyz=0$). A collection of such boundary divisors intersect if and only if the
corresponding root subsystems are pairwise orthogonal or nested. If $\Theta_1,\ldots,\Theta_k$ are a collection of such
root subsystems, we say the corresponding boundary stratum of $\oY$ is of type $\Theta_1 \cdots \Theta_k$. We write
$B_{\Theta_1 \cdots \Theta_k}$ for the sum of the boundary strata of the given type---thus, for instance, $B_{2A_1A_2^3}$
is the sum of the curves formed by the intersections of 2 $A_1$ divisors and 1 $A_2^3$ divisor.

Our first main result is a complete description of the cones of $W(E_6)$-invariant effective divisors and curves on
$\oY$.

\begin{theorem}[{\cref{thm:eff_Y,thm:mori_Y}}] \label{thm:eff_mori_Y_intro}
    \begin{enumerate}
        \item The $W(E_6)$-invariant effective cone of $\oY(E_6)$ is the closed cone spanned by the sum $B_{A_1}$ of the
            boundary divisors of type $A_1$, and the sum $B_{A_2^3}$ of the boundary divisors of type $A_2^3$.
        \item The $W(E_6)$-invariant cone of curves of $\oY(E_6)$ is the closed cone spanned by the sum $B_{3A_1}$ of
            the curves of type $3A_1$, and the sum $B_{2A_1A_2^3}$ of the curves of type $2A_1A_2^3$. In particular, a
            $W(E_6)$-invariant divisor on $\oY(E_6)$ is nef (resp. ample) if it intersects the curves of type $3A_1$ and
            $2A_1A_2^3$ nonnegatively (resp. positively).
    \end{enumerate}
\end{theorem}

As a corollary, we also obtain a complete description of the $W(E_6)$-invariant nef cone of $\oY$.

\begin{corollary}[{\cref{cor:nef_Y}}] \label{thm:nef_Y_intro}
    The $W(E_6)$-invariant nef cone of $\oY$ is spanned by $B_{A_1}+3B_{A_2^3}$ and $B_{A_1}+B_{A_2^3}$.
\end{corollary}

Recall the stable base locus of a $\bQ$-divisor $D$ is the set $\bB(D) = \bigcap_{m \geq 0} \text{Bs}(mD)$, where $m$ is
taken to be sufficiently divisible so that $mD$ is a $\bZ$-divisor, and $\text{Bs}(mD)$ is the set-theoretic base locus
of $mD$. The variety $\oY$ is rational, and linear, numerical, and homological equivalence coincide on $\oY$ (see
\cite{colomboChowGroupModuli2004} and \cite[Theorem 1.9]{schockQuasilinearTropicalCompactifications2021}). Therefore,
there is a well-behaved decomposition of the pseudoeffective cone of $\oY$ into chambers determined by the stable base
loci of the corresponding divisors. We completely describe this stable base locus decomposition for the
$W(E_6)$-invariant effective cone of $\oY$.

For two $\bQ$-divisors $D_1,D_2$, we write $(D_1,D_2]$ for the set of divisors of the form $\alpha D_1 + \beta D_2$ with
$\alpha \geq 0$, $\beta > 0$, and similarly for $[D_1,D_2)$, $(D_1,D_2)$, $[D_1,D_2]$. A $\bQ$-divisor $D$ is in
$(D_1,D_2]$ if and only if it lies in the cone spanned by the rays through $D_1$ and through $D_2$, but does not lie on
the ray through $D_1$.

\begin{theorem}[{\cref{thm:sbl_Y}}] \label{thm:sbl_Y_intro}
    Let $D$ be a $W(E_6)$-invariant effective $\bQ$-divisor on $\oY$.
    \begin{enumerate}
        \item If $D \in [B_{A_1}+3B_{A_2^3},B_{A_1}+B_{A_2^3}]$ then $D$ is nef and semi-ample. In particular, $\bB(D) =
            \emptyset$.
        \item If $D \in (B_{A_1}+3B_{A_2^3},B_{A_2^3}]$, then $\bB(D) = B_{A_2^3}$ is the sum of the $A_2^3$ divisors.
        \item If $D \in (B_{A_1}+B_{A_2^3},5B_{A_1}+3B_{A_2^3}]$, then $\bB(D)=B_{2A_1}$ is the sum of the $2A_1$ surfaces.
        \item If $D \in (5B_{A_1}+3B_{A_2^3},B_{A_1}]$, then $\bB(D)=B_{A_1}$ is the sum of the $A_1$ divisors.
    \end{enumerate}
\end{theorem}

\cref{thm:sbl_Y_intro} is pictured in \cref{fig:sbl}.

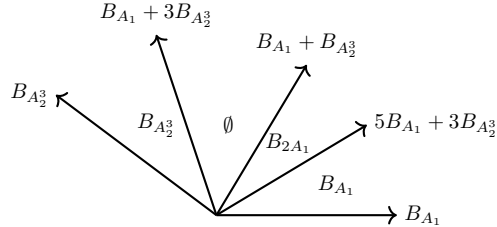
\begin{figure}[htpb]
    \centering
     \begin{tikzpicture}[scale=.8, transform shape]
       \draw[thick,->] (0,0) -- (3,0);
       \node[right] at (3,0) {$B_{A_1}$};
       \draw[thick,->] (0,0) -- (-4/1.5,3/1.5);
       \node[left] at (-4/1.5,3/1.5) {$B_{A_2^3}$};
       \draw[thick,->] (0,0) -- (-1,3);
       \node[above] at (-1,3) {$B_{A_1}+3B_{A_2^3}$};
       \draw[thick,->] (0,0) -- (1.5,2.5);
       \node[above] at (1.5,2.5) {$B_{A_1}+B_{A_2^3}$};
       \draw[thick,->] (0,0) -- (5/2,3/2);
       \node[right] at (5/2,3/2) {$5B_{A_1}+3B_{A_2^3}$};
       \node[] at (1.2,1.2) {$B_{2A_1}$};
       \node[] at (2,.5) {$B_{A_1}$};
       \node[] at (-1,1.5) {$B_{A_2^3}$};
       \node[] at (0.2,1.5) {$\emptyset$};
     \end{tikzpicture}
    \caption{The stable base locus decomposition of the $W(E_6)$-invariant effective cone of $\oY$.}%
    \label{fig:sbl}
\end{figure}

For an effective $\bQ$-divisor $D$ on $\oY$, let $\oY(D) = \Proj \bigoplus_{m \geq 0} H^0(\oY,mD)$ be the associated
projective model (if it exists), where the sum is taken over $m$ sufficiently divisible so that $mD$ is a $\bZ$-divisor.
We completely describe the $W(E_6)$-invariant birational models of $\oY$ appearing in the stable base locus
decomposition of the $W(E_6)$-invariant effective cone of $\oY$.

\begin{theorem}[{\cref{thm:models}}] \label{thm:models_intro}
    Let $D$ be a $W(E_6)$-invariant effective $\bQ$-divisor on $\oY$.
    \begin{enumerate}
        \item If $D \in (B_{A_1}+3B_{A_2^3},B_{A_1}+B_{A_2^3})$, then $D$ is ample, and $\oY(D) = \oY$.
        \item If $D \in [B_{A_1}+3B_{A_2^3},B_{A_2^3})$, then $\oY(D)$ is the GIT moduli space $\oM$ of marked cubic
            surfaces. The morphism $\pi : \oY \to \oM$ is given by contracting the 40 $A_2^3$ divisors of $\oY$ to
            singular points, each locally isomorphic to the cone over the Segre embedding of $(\bP^1)^3$.
        \item If $D$ is a multiple of $B_{A_1}+B_{A_2^3}$, then $D$ is semi-ample and the morphism $\phi : \oY \to
            \oY(D) =: \oW$ associated to $\lvert mD \rvert$ for $m \gg 0$ is given by contracting each $2A_1$ surface $S
            \cong Bl_4\bP^2$ to a singular line via the strict transform of the linear system of conics on $\bP^2$
            passing through the 4 blown up points.
        \item If $D \in (B_{A_1}+B_{A_2^3},B_{A_1})$, then $\oY(D)$ is the $D$-flip $\oX$ of the small
            contraction $\phi : \oY \to \oW$ of part (3).
        \item If $D$ is a multiple of $B_{A_1}$ or $B_{A_2^3}$, then $\oY(D)$ is a point.
    \end{enumerate}
\end{theorem}

In \cref{thm:models_intro}, the GIT moduli space $\oM$ of marked cubic surfaces is classically known; it is the natural
$W(E_6)$-cover of the GIT moduli space of cubic surfaces induced by choosing a marking, see for instance \cite[Section
2]{dolgachevComplexBallUniformization2005}. The contraction $\pi : \oY \to \oM$ is explicitly constructed in Naruki's
original work \cite{narukiCrossRatioVariety1982}. The space $\oM$ also has an interpretation as the K-moduli space of
weighted marked cubic surfaces $(S,cB)$ with weight $0 < c < 1/9$, cf. \cite{zhaoCompactificationsModuliPezzo2023,
schockModuliWeightedStable2023}, and as the Baily-Borel compactification of Allcock, Carlson, and Toledo's ball
quotient model of the moduli space of marked cubic surfaces \cite{allcockComplexHyperbolicGeometry2002,
gallardoGeometricInterpretationToroidal2021}.  Similarly, the space $\oY$ has an intrepretation as the moduli space
of KSBA weighted stable marked cubic surface $(S,cB)$, with weight $1/9 < c \leq 1/4$, as mentioned previously, see
\cite{gallardoGeometricInterpretationToroidal2021, schockModuliWeightedStable2023}, and as the toroidal
compactification of Allcock, Carlson, and Toledo's ball quotient model of the moduli space of marked cubic surfaces
\cite{allcockComplexHyperbolicGeometry2002, gallardoGeometricInterpretationToroidal2021}. The other birational
models in \cref{thm:models_intro} appear to be new.

Our second moduli space of interest is the moduli space $\wY = \wY(E_6)$ of KSBA stable marked cubic surfaces
\cite{hackingStablePairTropical2009}. The space $\wY$ is the blowup of $\oY$ along all intersections of $A_1$ divisors,
in increasing order of dimension. The boundary of $\wY$ consists of 5 types of irreducible divisors, which we label by
$a,b,a_2,a_3,a_4$. These are, respectively, the strict transforms of the $A_1$ and $A_2^3$ divisors, and the (strict
transforms of the) exceptional divisors over the intersections of $2,3,4$ $A_1$ divisors. The moduli space $\wY$ also
has 45 Eckardt divisors, parameterizing marked cubic surfaces with Eckardt points. We label the Eckardt divisors by type
$e$, and we choose not to consider them as part of the boundary. We follow similar notation for strata of $\wY$ as for
$\oY$.

\begin{theorem}[{\cref{thm:eff_cone,thm:mori_cone}}] \label{thm:eff_cones_intro}
    \begin{enumerate}
        \item The $W(E_6)$-invariant effective cone of $\wY(E_6)$ is the closed cone spanned by the $W(E_6)$-invariant
            boundary divisors $B_a$, $B_{a_2}$, $B_{a_3}$, $B_{a_4}$, and $B_b$.
        \item The $W(E_6)$-invariant cone of curves of $\wY(E_6)$ is the closed cone spanned by the $W(E_6)$-invariant
            curves of types
            \[
                aa_2a_3, aa_2a_4, aa_3a_4, a_2a_3a_4, aa_2b, aa_3b, a_2a_3b, aa_2e.
            \]
            (These are the boundary curves of $\wY(E_6)$, together with the curves of type $aa_2e$ formed by the
            intersection of one type $a$ divisor, one type $a_2$ divisor, and one Eckardt divisor.) In particular, a
            $W(E_6)$-invariant divisor on $\wY(E_6)$ is nef (resp. ample) if it intersects the curves of the above types
            nonnegatively (resp. positively). Explicitly, the $W(E_6)$-invariant nef cone of $\wY(E_6)$ is the closed
            cone spanned by the rays through the following divisor classes.
            \begin{align*}
                5B_{a} + 6B_{a_2} + 7B_{a_3} + 8B_{a_4} + 3B_b, \\
                3B_{a} + 6B_{a_2} + 7B_{a_3} + 8B_{a_4} + 3B_b, \\
                B_{a} + 2B_{a_2} + 3B_{a_3} + 2B_{a_4} + 3B_b, \\
                B_{a} + 2B_{a_2} + 3B_{a_3} + 4B_{a_4} + 3B_b, \\
                B_{a} + 2B_{a_2} + 3B_{a_3} + 4B_{a_4} + B_b, \\
                B_{a} + 2B_{a_2} + 3B_{a_3} + 2B_{a_4} + 2B_b.
            \end{align*}
    \end{enumerate}
\end{theorem}

The $W(E_6)$-invariant effective cone of $\wY$ is 5-dimensional and its stable base locus decomposition appears to be
quite complicated in comparison to that of $\oY$. We describe an interesting slice of this decomposition, by describing
the log minimal models of the pair $(\wY,cB+dE)$, with respect to the divisor $K_{\wY} + cB + dE$, where $B$ is the sum
of the boundary divisors and $E$ is the sum of the Eckardt divisors.

\begin{theorem}[{\cref{thm:log_mmp}}] \label{thm:log_mmp_intro}
    Fix $0 \leq c \leq 1$ and $0 \leq d \leq 2/3$. Then the pair $(\wY,cB+dE)$ has log canonical singularities. The log
    canonical models of $(\wY,cB+dE)$ are as follows.
    \begin{enumerate}
        \item If $4c + 25d < 1$, then $K_{\wY} + cB + dE$ is not effective.
        \item If $4c+25d = 1$, then the log canonical model of $(\wY,cB+dE)$ is a point.
        \item If $2c+12d \leq 1$ and $4c+25d > 1$ then the log canonical model of $(\wY,cB+dE)$ is the GIT moduli space
            $\oM$ of marked cubic surfaces.
        \item If $c+4d \leq 1$ and $2c+12d > 1$, then the log canonical model of $(\wY,cB+dE)$ is Naruki's
            compactification $\oY$.
        \item If $c+3d \leq 1$ and $c+4d > 1$, then the log canonical model of $(\wY,cB+dE)$ is the blowup $\oY_1$ of
            $\oY$ along the intersections of 4 $A_1$ divisors.
        \item If $c+2d \leq 1$ and $c+3d > 1$, then the log canonical model of $(\wY,cB+dE)$ is the blowup $\oY_2$ of
            $\oY_1$ along the strict transforms of the intersections of 3 $A_1$ divisors.
        \item If $c+2d > 1$, then $K_{\wY}+cB+dE$ is ample, so $(\wY,cB+dE)$ is already its own log canonical model.
            Recall this is the blowup of $\oY_2$ along the strict transforms of the intersections of 2 $A_1$ divisors.
    \end{enumerate}
\end{theorem}

\cref{thm:log_mmp_intro} is pictured in \cref{fig:log_mmp}. Note that when $d=0$, the log canonical model of $(\wY,B)$
is Naruki's compactification $\oY$. (This is also shown in \cite{hackingStablePairTropical2009}.) The log minimal model
program for $(\oY,cB_{\oY})$ is captured by the bottom line of \cref{fig:log_mmp}; this is also captured by the left
half of \cref{fig:sbl}.

The analogue of \cref{thm:log_mmp_intro} for moduli of curves is sometimes called the Hassett-Keel program, and has seen
significant interest as a method to interpolate between different geometrically interesting compactifications of the
moduli space of curves, see for instance \cite{hassettLogCanonicalModels2009, hassettLogMinimalModel2013}. Recently, a
version of the Hassett-Keel program for moduli of K3 surfaces, dubbed the \emph{Hassett-Keel-Looijenga program}, has
seen intensive study, see for instance \cite{ascherKstabilityBirationalModels2023,
lazaGITBailyBorelCompactification2018, lazaBirationalGeometryModuli2019}. \cref{thm:log_mmp_intro} is another
higher-dimensional generalization of the Hassett-Keel program.

\begin{figure}[htpb]
    \centering
     \begin{tikzpicture}[scale=0.7, transform shape]
       \draw[thick] (0,0) -- (10,0);
       \draw[thick] (0,0) -- (0,8);
       \draw[thick] (0,6) -- (10,0);
       \node[left] at (0,8) {$2/3$};
       \node[left] at (0,6) {$1/2$};
       \node[left] at (0,4) {$1/3$};
       \node[left] at (0,3) {$1/4$};
       \node[left] at (0,1) {$1/12$};
       \node[left] at (0,12/25) {$1/25$};
       \draw[thick] (0,4) -- (10,0);
       \draw[thick] (0,3) -- (10,0);
       \draw[thick] (0,1) -- (5,0);
       \draw[thick] (0,12/25) -- (10/4,0);
       \draw[thick] (10,0) -- (10,8);
       \draw[thick] (0,8) -- (10,8);
       \node[] at (6,6) {$\wY$};
       \node[] at (4,3) {$\oY_2$};
       \node[] at (3,2.5) {$\oY_1$};
       \node[] at (2,1.5) {$\oY$};
       \node[] at (1,.5) {$\oM$};
       \node[right] at (10,0) {$c$};
       \node[above] at (0,8) {$d$};
       \node[below] at (10,0) {$1$};
       \node[below] at (5,0) {$1/2$};
       \node[below] at (10/4,0) {$1/4$};

     \end{tikzpicture}
    \caption{The log minimal model program for $\wY$.}%
    \label{fig:log_mmp}
\end{figure}
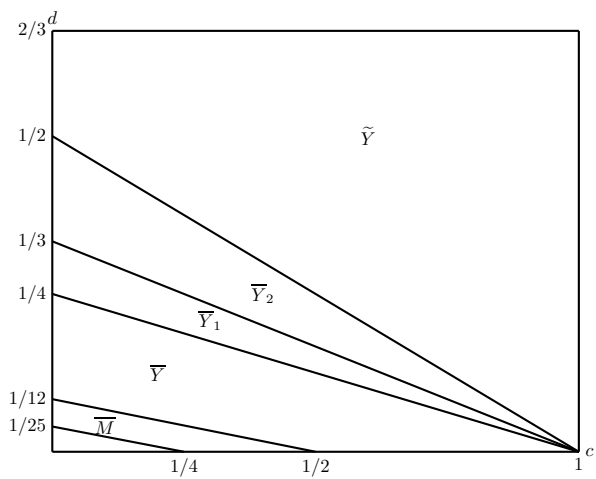

\subsection{Context}

The study of the birational geometry of moduli spaces is a topic of great interest, see for instance
\cite{keelContractibleExtremalRays1996, gibneyAmpleConeOverline2002, castravetOverlineNotMori2015,
castravetBlownupToricSurfaces2023, fedorchukAmpleDivisorsModuli2011, hassettLogCanonicalModels2009,
hassettLogMinimalModel2013, hassettEffectiveConeModuli2002, moonMoriProgramOverline2015, moonMoriProgramOverline2017,
moonLogCanonicalModels2013, chenMoriProgramModuli2011, coskunAmpleConeKontsevich2009, coskunAmpleConeModuli2016,
coskunEffectiveConeModuli2017}.

An ideal study of the birational geometry of a moduli space $\oM$ would include a complete description of the cones of
effective divisors and curves on $\oM$, a decomposition of the effective cone of divisors into chambers according to the
stable base loci of the corresponding divisors, and a description of the birational models of $\oM$ resulting from
divisors in each chamber. However, this is in general a quite difficult problem, even for the moduli space $\oM_{0,n}$
of stable $n$-pointed genus zero curves.

The effective cone of divisors of $\oM_{0,n}$ is in general not even polyhedral
\cite{castravetBlownupToricSurfaces2023}, and the F-conjecture, which asserts that any extremal ray of the effective
cone of curves of $\oM_{0,n}$ is generated by a one-dimensional boundary stratum, has been open for nearly 30 years
\cite{keelContractibleExtremalRays1996}. (A solution to the F-conjecture for $\oM_{0,n}$ would also imply the analogous
conjecture for $\coM_{g,n}$ \cite{gibneyAmpleConeOverline2002}.) The $S_n$-invariant birational geometry of $\oM_{0,n}$ is
somewhat better understood---the cone of $S_n$-invariant effective divisors is generated by the $S_n$-invariant boundary
divisors \cite{keelContractibleExtremalRays1996}, the $S_n$-invariant F-conjecture is known for $n \leq 35$
\cite{fedorchukSymmetricFconjectureLeq2020}, and a complete description of the stable base locus decomposition of the
$S_n$-invariant effective cone of divisors, and the corresponding birational models, is known for $n \leq 7$
\cite{moonMoriProgramOverline2015, moonMoriProgramOverline2017}. Our moduli spaces of interest $\oY$ and $\wY$ are natural
higher-dimensional generalizations of $\oM_{0,n}$ (cf. \cite{hackingStablePairTropical2009}), and our main results
generalize the study of the $S_n$-invariant birational geometry of $\oM_{0,n}$ to higher dimensions. To our knowledge the
present article is one of the first nontrivial examples studying in detail the birational geometry of a KSBA moduli
space of higher-dimensional stable pairs.

\subsection{Outline}

The paper is organized as follows. In \cref{sec:naruki} we describe in depth the geometry of Naruki's compactification
$\oY$; in particular, we give a detailed description of the boundary and Eckardt divisors and their intersections, and
gather the necessary results concerning intersection theory on $\oY$. In \cref{sec:bir_naruki}, we use these results to
prove the main theorems on the birational geometry of $\oY$, namely,
\cref{thm:eff_mori_Y_intro,thm:sbl_Y_intro,thm:models_intro}. In \cref{sec:wY} we carry out the similar analysis of the
geometry of the moduli space $\wY$, in preparation for \cref{sec:bir_wY} where we prove the main theorems on the birational
geometry of $\wY$, namely, \cref{thm:eff_cones_intro,thm:log_mmp_intro}.

\section{Naruki's compactification} \label{sec:naruki}

In this section we gather the necessary results concerning divisors and intersection theory on Naruki's compactification
$\oY$ of the moduli space $Y$ of marked cubic surfaces \cite{narukiCrossRatioVariety1982}, in order to study the
birational geometry of $\oY$ in the following section.

Let us recall Naruki's construction of $\oY$. Let $Z$ be the moduli space of triples $(S,m,D)$ consisting of a smooth
cubic surface $S$, a marking $m$ of $S$, and an anticanonical cycle $D$ on $S$ consisting of three $(-1)$-curves
$D_1,D_2,D_3$ forming a triangle. We assume an orientation of $D$ has been fixed; this amounts to fixing an isomorphism
$\Pic^0(D) \cong \bC^*$ (cf. \cite[Lemma 1.6]{friedmanGeometryAnticanonicalPairs2016}). The orthogonal complement
$D^{\perp} := \langle D_1,D_2,D_3 \rangle^{\perp}$ in $\Pic S$ is isomorphic to the $D_4$ root lattice $\Lambda(D_4)$
\cite[Appendix by E. Looijenga]{narukiCrossRatioVariety1982}. It follows that $\Hom(D^{\perp},\Pic^0(D)) \cong
\Hom(\Lambda(D_4),\bC^*)$ is identified with the algebraic torus $T_{\Lambda(D_4)} \cong (\bC^*)^4$ with character
lattice $\Lambda(D_4)$. The period map $Z \to \Hom(D^{\perp},\Pic^0(D))$ sending $(S,m,D)$ to $(\alpha \mapsto
\alpha\vert_D)$ induces an isomorphism of $Z$ with the complement in $T_{\Lambda(D_4)}$ of the \emph{root hypertori}
$H_{\alpha} = \{\chi \mid \chi(\alpha)=1\}$, for $\alpha$ a positive root of $D_4$. (Note $H_{\alpha}=H_{-\alpha}$.) For
more details we refer to \cite[Appendix]{narukiCrossRatioVariety1982}, see also related ideas in
\cite{looijengaRationalSurfacesAntiCanonical1981, looijengaArtinGroupsFundamental2008,
grossModuliSurfacesAnticanonical2015, friedmanGeometryAnticanonicalPairs2016}.

There is a natural compactification $\oZ$ of $Z$, the so-called \emph{minimal wonderful compactification} (cf.
\cite{deconciniProjectiveWonderfulModels2018}), which Naruki describes as follows \cite{narukiCrossRatioVariety1982}.
Let $X(\Sigma)$ be the $D_4$ Coxeter toric variety, i.e., the toric variety with torus $T_{\Lambda(D_4)}$ whose maximal
cones are the $D_4$ Weyl chambers (cf. \cite{batyrevFunctorToricVarieties2011}). Connected components of intersections
of the root hypertori $H_{\alpha}$ in $T_{\Lambda(D_4)}$ correspond to root subsystems of $D_4$ whose Dynkin diagrams
look like Dynkin diagrams obtained by deleting vertices from the extended Dynkin diagram of $D_4$ \cite[Theorem
7.17]{alexeevADESurfacesTheir2020}.

Blowing up the $D_4$ toric variety $X(\Sigma)$ in the closures of the subvarieties of $T_{\Lambda(D_4)}$ corresponding
to irreducible root subsystems of $D_4$ produces a smooth projective variety $\oZ$ such that $\oZ \setminus Z$ is a
simple normal crossings divisor. Explicitly, $\oZ$ is obtained from $X(\Sigma)$ by blowing up 1 point (the identity of
$T_{\Lambda(D_4)}$) corresponding to the $D_4$ root system itself, then 12 curves corresponding to $A_3 \subset D_4$
root subsystems, then 16 surfaces corresponding to $A_2 \subset D_4$ root subsystems.

Naruki shows that the divisors of $\oZ$ over the $A_3$ curves are each isomorphic to $\bP^1 \times \oM_{0,5}$, and each
can be blown down via the second projection $\bP^1 \times \oM_{0,5} \to \oM_{0,5}$ \cite{narukiCrossRatioVariety1982}.
The resulting variety is his ``cross-ratio variety'' $\oY$, compactifying the moduli space of marked cubic surfaces. (On
the interiors, the map $\oZ \to \oY$ is identified with the natural map forgetting the anticanonical cycle on the marked
cubic surface.)

\subsection{Divisors on $\oY$} \label{sec:div_Y}

There are three natural types of divisors of interest on Naruki's compactification.
\begin{enumerate}
    \item 36 $A_1$ divisors (corresponding to $A_1 \subset E_6$ root subsystems), parameterizing marked cubic surfaces
        with $A_1$ singularities. These appear on $\oY$ as the images of the 12 $A_1 \subset D_4$ root hypersurfaces of
        $\oZ$, and 24 of the 48 boundary divisors of $X(\Sigma)$, see \cite{narukiCrossRatioVariety1982,
        colomboChowGroupModuli2004}.
    \item 40 $A_2^3$ divisors (corresponding to $A_2^3 = A_2 \times A_2 \times A_2 \subset E_6$ root subsystems),
        parameterizing marked cubic surfaces which are the reducible union of three planes meeting transversally (e.g.,
        $xyz=0$ in $\bP^3$). These appear on $\oY$ as the images of the 16 $A_2 \subset D_4$ divisors of $\oZ$, and the
        remaining 24 of the 48 boundary divisors of $X(\Sigma)$, again see \cite{narukiCrossRatioVariety1982,
        colomboChowGroupModuli2004}.
    \item 45 Eckardt divisors, parameterizing marked cubic surfaces with Eckardt points. These appear on $\oY$ as the
        images of the $D_4$ divisor of $\oZ$ (the strict transform of the exceptional divisor over the identity of the
        torus $T_{\Lambda(D_4)}$), and 44 more divisors of $\oZ$ obtained from hypersurfaces of $T_{\Lambda(D_4)}$ which
        are explicitly written down in \cite{narukiCrossRatioVariety1982}. The Eckardt divisors are labeled as follows.
        Suppose $S$ is a smooth cubic surface obtained by blowing up 6 points $p_1,\ldots,p_6$ in $\bP^2$ in general
        position. Then the lines on $S$ are labeled by $e_i$ for $i=1,\ldots,6$ (the exceptional divisor over $p_i$),
        $c_i$ for $i=1,\ldots,6$ (the conic passing through $p_j$ for $j \neq i$), and $\ell_{ij}$ for $ij \subset [6]$
        (the line through $p_i$ and $p_j$). The possible Eckardt points on $S$ are given by two types of triples of
        lines: 30 of the form $\{e_i,c_j,\ell_{ij}\}$, and 15 of the form $\{\ell_{ij},\ell_{kl},\ell_{mn}\}$. These
        triples give the labels of the 45 Eckardt divisors. We say that two Eckardt divisors have a common line if the
        corresponding triples have a common line. (In \cite{narukiCrossRatioVariety1982,
        vangeemenLinearSystemNaruki2002, colomboChowGroupModuli2004}, Eckardt divisors are instead called
        \emph{tritangent} divisors.)
\end{enumerate}
We write $B_{A_1}$ for the sum of the $A_1$ divisors and $B_{A_2^3}$ for the sum of the $A_2^3$ divisors.

From Naruki's construction of $\oY$ one is able to read off a wealth of information concerning these divisors and their
intersections, as summarized below. For more information we refer to \cite{narukiCrossRatioVariety1982,
vangeemenLinearSystemNaruki2002, colomboChowGroupModuli2004}.

\begin{notation}
    Recall that the divisors of $\oM_{0,n}$ are labeled by $D_I = D_{I,I^c}$, with $\lvert I \rvert, \lvert I^c \rvert
    \geq 2$, parameterizing stable genus $0$ curves with 2 irreducible components, such that the points from $I$ are on
    one component, and the remaining points on the other component. When we are not concerned with the particular
    boundary divisor, we write $D_{\lvert I \rvert}$ to denote a boundary divisor of the form $D_I$, and we write
    $B_{\lvert I \rvert}$ for the sum of all such divisors.

    We also recall the Keel-Vermeire divisors on $\oM_{0,6}$ \cite{vermeireCounterexampleFultonConjecture2002}. For
    $(ij)(kl)(mn) \in S_6$, let $\oM_{0,6} \to \coM_3$ be the morphism sending a genus $0$ curve with 6 marked points to
    a genus 3 curve by identifying the marked points $i$ and $j$, $k$ and $l$, and $m$ and $n$. The Keel-Vermeire
    divisor associated to $(ij)(kl)(mn)$ is the pullback of the hyperelliptic locus of $\coM_3$. There are 15
    Keel-Vermeire divisors on $\oM_{0,6}$, and they are effective divisors which cannot be written as effective sums of
    boundary divisors.
\end{notation}

\begin{proposition} \label{prop:A1_naruki}
    There are 36 $A_1$ divisors on $\oY$. Let $D = D_{A_1}$ be an $A_1$ divisor. Then $D \cong \oM_{0,6}$, and the
    nonempty intersections of $D$ with the Eckardt divisors and the other boundary divisors of $\oY$ are as follows.
    \begin{enumerate}
        \item $D$ intersects 15 other $A_1$ divisors, $D_{A_1'}$ for $A_1'$ orthogonal to $A_1$, in the 15 divisors
            $D_{ij}$ on $\oM_{0,6}$.
        \item $D$ intersects 10 $A_2^3$ divisors, $D_{A_2^3}$ for $A_1 \subset A_2^3$, in the 10 divisors $D_{ijk}$ on
            $\oM_{0,6}$.
        \item $D$ intersects every Eckardt divisor $D_e$.
            \begin{enumerate}
                \item If the lines of $D_e$ are lines on the $A_1$-nodal cubic surface generically parameterized by $D$,
                    then $D_e$ restricts to a Keel-Vermeire divisor on $\oM_{0,6}$. There are 15 such divisors.
                \item Otherwise, there is a unique second Eckardt divisor $D_e'$ and a unique second $A_1$ divisor
                    $D_{A_1'}$, such that
                    \[
                        D_e \cap D_{A_1} = D_e' \cap D_{A_1'} = D_{A_1} \cap D_{A_1'} = D_{ij}.
                    \]
                    There are 15 such pairs.
            \end{enumerate}
    \end{enumerate}
    The class of the restriction of $D_{A_1}$ to itself is
    \[
        \frac{-B_2 - 3B_3}{5}.
    \]
\end{proposition}

\begin{proof}
    This all follows from Naruki's explicit construction of $\oY$ \cite{narukiCrossRatioVariety1982}, also described in
    greater detail in \cite{sekiguchiCrossRatioVarieties1994, vangeemenLinearSystemNaruki2002,
    colomboChowGroupModuli2004, hackingStablePairTropical2009}. The only part not explained in detail in \textit{loc.
    cit.} is the intersections with the Eckardt divisors which give the Keel-Vermeire divisors on $\oM_{0,6}$. For this, it
    suffices to consider one particular example.  Suppose $D$ is the $A_1$ divisor obtained by blowing up 6 points on a
    conic. Then the three lines $\ell_{ij}$, $\ell_{kl}$, $\ell_{mn}$ intersect in a point precisely when the marked genus
    $0$ curve given by the conic through the 6 points lies in the pullback of the hyperelliptic locus of $\coM_3$. Thus the
    intersection with the Eckardt divisor given by $\{\ell_{ij},\ell_{kl},\ell_{mn}\}$ is the Keel-Vermeire divisor of type
    $(ij)(kl)(mn)$.
\end{proof}

\begin{proposition} \label{prop:A23_naruki}
    There are 40 $A_2^3$ divisors on $\oY$. Let $D = D_{A_2^3}$ be an $A_2$ divisor. Then $D \cong (\bP^1)^3$, and the
    nonempty intersections of $D$ with the Eckardt divisors and the other boundary divisors of $\oY$ are as follows.
    \begin{enumerate}
        \item $D$ intersects 9 $A_1$ divisors, $D_{A_1}$ for $A_1 \subset A_2^3$, in the 9 divisors $p \times \bP^1
            \times \bP^1$, $\bP^1 \times p \times \bP^1$, $\bP^1 \times \bP^1 \times p$, for $p \in \{0,1,\infty\}$.
        \item $D$ is disjoint from the other $A_2^3$ divisors.
        \item $D$ intersects 18 Eckardt divisors in 18 smooth hypersurfaces, 6 each of the classes $h_1+h_2$, $h_1+h_3$,
            $h_2+h_3$, where $h_1,h_2,h_3$ are the classes of the 3 rulings on $(\bP^1)^3$. If one fixes the $k$th
            $\bP^1$, then the hypersurfaces of class $h_i+h_j$ are determined uniquely by specifying that they pass
            through the three distinct points $p_1 \times q_1, p_2 \times q_2, p_3 \times q_3$ in the remaining $\bP^1
            \times \bP^1$, where $p_l,q_l \in \{0,1,\infty\}$.
    \end{enumerate}
    The class of the restriction of $D_{A_2^3}$ to itself is
    \[
        -h_1-h_2-h_3.
    \]
\end{proposition}

\begin{proof}
    Again this all follows from Naruki's construction \cite{narukiCrossRatioVariety1982}, as described in detail in
    \cite{sekiguchiCrossRatioVarieties1994, vangeemenLinearSystemNaruki2002, colomboChowGroupModuli2004,
    hackingStablePairTropical2009} (see also \cite[Proposition
    5.4]{casalaina-martinNonisomorphicSmoothCompactifications2022}).
\end{proof}

\begin{proposition} \label{prop:eckardt_naruki}
    There are 45 Eckardt divisors on $\oY$. Let $D=D_e$ be an Eckardt divisor. Then $D$ is isomorphic to the minimal
    wonderful compactification $\oX(D_4)$ of the $D_4$ hyperplane arrangement $x_i = \pm x_j$, $i \neq j$, in $\bP^3$,
    obtained by blowing up the 12 points corresponding to $A_3 \subset D_4$ root subsystems, followed by the 16 lines
    corresponding to $A_2 \subset D_4$ hyperplanes. The intersections of $D$ with the other Eckardt divisors and the
    boundary divisors of $\oY$ are as follows.
    \begin{enumerate}
        \item $D$ intersects every $A_1$ divisor.
            \begin{enumerate}
                \item If the lines of $D$ are lines on the $A_1$-nodal cubic surface generically parameterized by
                    $D_{A_1}$, then $D_{A_1}$ restricts to one of the $D_4$ hyperplanes on $D$. There are 12 such
                    intersections, each isomorphic to $Bl_3\bP^2$.
                \item Otherwise, there is a unique second $A_1$ divisor $D_{A_1'}$ such that
                    \begin{align*}
                        D \cap D_{A_1} = D \cap D_{A_1'} = D_{A_1} \cap D_{A_1'}.
                    \end{align*}
                    On $D$, this intersection is the strict transform of the exceptional divisor over one of the
                    blown up points. There are 12 such pairs of intersections. Each intersection is isomorphic to
                    $Bl_4\bP^2 \cong \oM_{0,5}$.
            \end{enumerate}
        \item $D$ intersects 16 $A_2^3$ divisors. These intersections occur precisely when the lines of $D$ all lie on
            the same one of the three planes in the reducible cubic surface generically parameterized by the given
            $A_2^3$ divisor. These intersections give the exceptional divisors over the 16 blown up lines on $D$. Each
            such intersection is isomorphic to $\bP^1 \times \bP^1$.
        \item $D$ intersects every other Eckardt divisor.
            \begin{enumerate}
                \item If $D'$ is an Eckardt divisor having no common lines with $D$, then there is a unique third
                    Eckardt divisor $D''$ such that the intersection of any 2 of $D,D',D''$ is the same as the
                    intersection of all 3. There are 16 such pairs, restricting in $D$ to the strict transforms of 16
                    smooth quadrics in $\bP^3$. Each such quadric is determined by one of the blown-up lines, by the
                    condition that the quadric passes through all of the blown-up points except for the three lying on
                    the chosen line \cite[Section 6.7]{vangeemenLinearSystemNaruki2002}. In particular, each such
                    intersection is isomorphic to the blowup of $\bP^1 \times \bP^1$ at 9 points $p \times q$ for $p,q
                    \in \{0,1,\infty\}$.
                \item If $D$ and $D'$ have a common line, then their intersection is reducible. On $D$ it looks like the
                    union of the strict transform of an $F_4$ hyperplane $x_i=0$ or $x_1 \pm x_2 \pm x_3 \pm x_4 = 0$,
                    together with the strict transform of the exceptional divisor over one of the blown up points which
                    is contained in the given $F_4$ hyperplane. There are 12 such intersections.
            \end{enumerate}
    \end{enumerate}
\end{proposition}

\begin{proof}
    This also follows from Naruki's construction \cite{narukiCrossRatioVariety1982}, as described in detail in
    \cite{colomboChowGroupModuli2004, vangeemenLinearSystemNaruki2002}. However, we note that part of the result is
    misstated in \cite[Section 6.7]{vangeemenLinearSystemNaruki2002}, where it is stated that if $D$ and $D'$ have a
    common line, then their intersection consists only of the strict transform $H$ of an $F_4$ hyperplane, rather than
    the union of $H$ with the strict transform $P$ of an exceptional divisor over a blown up point. However, this is
    inconsistent with \cref{prop:A1_naruki}. Indeed, suppose $P = D \cap D_{A_1} = D \cap D_{A_1'}$, as in
    \cref{prop:A1_naruki} and part 1(b) of the current proposition. Then by \cref{prop:A1_naruki}, $D \cap D_{A_1} \cap
    D_{A_1'} = D' \cap D_{A_1} \cap D_{A_1'}$. This is of course false unless $D \cap D'$ contains $P$. The corrected
    description follows from Naruki's construction and his explicit description of the Eckardt divisors starting from
    their equations in $Z \subset T_{\Lambda(D_4)} \cong (\bC^*)^4$ \cite[Section 8]{narukiCrossRatioVariety1982}.
    Namely, Naruki puts coordinates $\lambda,\mu,\nu,\rho$ on $(\bC^*)^4$ corresponding to a basis of simple roots of
    $D_4$, and writes down explicit equations in terms of these coordinates for each Eckardt divisor \cite[Table
    3]{narukiCrossRatioVariety1982}. For instance, two such Eckardt divisors are given by $\{\lambda\nu = 1\}$ and
    $\{\lambda = \nu\}$. The intersection of these two divisors has two connected components, $\{\lambda = \nu = 1\}$,
    $\{\lambda = \nu = -1\}$. Neither of these components is blown up by $\oZ(D_4) \to X(\Sigma)$, since the root
    subsystem of $D_4$ spanned by the roots corresponding to $\lambda$ and $\nu$ is reducible, an $A_1^2$ subsystem.
    Tracing these intersections through Naruki's construction $\oY \leftarrow \oZ(D_4) \rightarrow X(\Sigma)$, one sees
    the correct description.
\end{proof}

\begin{example} \label{ex:eckardt_naruki}
    Let $S$ be a smooth cubic surface obtained as the blowup of 6 points in $\bP^2$, and write $\Pic S = \langle
    h,e_1,\ldots,e_6 \rangle$, where $h$ is the pullback of the hyperplane class and $e_i$ is the class of the
    exceptional divisor over the $i$th point.

    Recall that the root system $E_6$ is given by
    \[
        E_6 = \{\alpha \in \Pic S \mid \alpha \cdot K_S = 0, \alpha^2 = -2\}.
    \]
    Explicitly, we label the positive roots of $E_6$ by $ij=e_i-e_j$, $ijk=h-e_i-e_j-e_k$, and
    $7=\beta=2h-e_1-\cdots-e_6$ (cf. \cite[Remark 4.9]{hackingStablePairTropical2009}). We identify $A_1$ subsystems
    with positive roots.

    We consider the Eckardt divisor $D_e \cong \oX(D_4)$ in $\oY$ corresponding to the triple of lines
    $\{e_5,c_6,\ell_{56}\}$.
    \begin{enumerate}
        \item The divisors $D_{A_1}$ such that the lines of $D_e$ are lines on the $A_1$-nodal cubic parameterized by
            $D_{A_1}$ are given by $A_1 = ij$ or $ij6$, for $ij \subset [4]=\{1,2,3,4\}$. (Here and elsewhere $ij =
            \{i,j\}$.) These cut out the strict transforms of the 12 $D_4$ hyperplanes on $D_e$.
        \item The remaining $D_{A_1}$ divisors come in pairs of orthogonal $A_1$'s, cutting out the strict transforms of
            the exceptional divisors of $\oX(D_4)$ over the 12 blown up points. For instance, if $A_1 = 7$ and $A_1' =
            56$, then $D_e \cap D_{A_1} = D_e \cap D_{A_1'}$ is the exceptional divisor over the point $(1:1:1:1)$.
        \item The Eckardt divisors sharing a line with $D_e$ intersect $D_e$ in the reducible union of the strict
            transform of an $F_4$ hyperplane and the strict transform of the exceptional divisor over one of the 12
            blown up points. For instance, if $D_e'$ is given by the triple of lines $\{e_6,c_5,\ell_{56}\}$, then $D_e
            \cap D_e'$ is the union of the hyperplane $\{x_1-x_2+x_3-x_4=0\}$ and the strict transform of the
            exceptional divisor over the point $(1:1:1:1)$. Note that this and the previous part are consistent with the
            fact that $D_e \cap D_{A_1} = D_e' \cap D_{A_1} = D_{A_1} \cap D_{A_1'}$ for $A_1=7$, $A_1'=56$, as expected
            from \cref{prop:A1_naruki}.
    \end{enumerate}
\end{example}

\subsection{Intersection-theoretic results on $\oY$}

The intersection theory of $\oY$ has been studied in depth in \cite{colomboChowGroupModuli2004}. Additional results,
including a complete presentation of its Chow ring, are obtained in \cite[Theorem
1.9]{schockQuasilinearTropicalCompactifications2021}. In addition to the descriptions of intersections and
self-intersections of boundary and Eckardt divisors above, we also have need of the following results.

\begin{proposition}[{\cite{colomboChowGroupModuli2004}}] \label{prop:int_Y}
    \begin{enumerate}
        \item We have $\Pic(\oY)^{W(E_6)} \cong \bZ^2$, generated by the sum $B_{A_1}$ of the $A_1$ divisors, and the
            sum $B_{A_2^3}$ of the $A_2^3$ divisors.
        \item The class of the sum $E$ of the Eckardt divisors on $\oY$ is given by
            \[
                E = \frac{25B_{A_1} + 27B_{A_2^3}}{4}.
            \]
        \item The canonical class of $\oY$ is given by
            \[
                K_{\oY} = \frac{-B_{A_1} + B_{A_2^3}}{4}.
            \]
    \end{enumerate}
\end{proposition}

Recall from the introduction that we label strata of $\oY$ by juxtaposition, so for instance a stratum of type
$2A_1A_2^3$ is a curve formed by the intersection of 2 $A_1$ and 1 $A_2^3$ divisors.

\begin{proposition} \label{prop:int_nums_Y}
    The 1-dimensional boundary strata of $\oY$ are of types $3A_1$ and $2A_1A_2^3$. The intersection numbers of
    $W(E_6)$-invariant boundary divisors on $\oY(E_6)$ with curves of these types are given in the following table.
    \begin{table}[htpb]
        \centering
        \caption{Intersection numbers on $\oY(E_6)$.}
        \label{tab:int_nums_Y}
        \begin{tabular}{| c || c | c |}
            \hline
                  & $3A_1$ & $2A_1A_2^3$ \\
                  \hline\hline
            $B_{A_1}$ & $-2$ & $3$ \\
            $B_{A_2^3}$ & $2$ & $-1$ \\
            \hline
        \end{tabular}
    \end{table}
\end{proposition}

\begin{proof}
    Let $D$ be an irreducible boundary divisor of type $A_1$. Using \cref{prop:A1_naruki}, we compute that
    \begin{align*}
        B_{A_1}\vert_{D} = \frac{4B_2 - 3B_3}{5} \;\; \text{ and } \;\; B_{A_2^3}\vert_{D} = B_3.
    \end{align*}
    The $3A_1$ curves contained in $D \cong \oM_{0,6}$ are precisely the intersections of two $D_2$ divisors on
    $\oM_{0,6}$; likewise, the $2A_1A_2^3$ curves contained in $D \cong \oM_{0,6}$ are precisely the intersections of
    one $D_2$ and one $D_3$ divisor on $\oM_{0,6}$.  The result now follows by standard intersection-theoretic
    computations on $\oM_{0,6}$, see \cite[Corollary 2.6]{moonMoriProgramOverline2015}.
\end{proof}

\begin{lemma} \label{lem:Zdiv}
    Let $D$ be a $\bQ$-divisor on $\oY$, and suppose that $D \cdot C \in \bZ$ for each 1-dimensional boundary stratum
    $C$ of $\oY$. Then $D$ is actually a $\bZ$-divisor on $\oY$.
\end{lemma}

\begin{proof}
    By \cite[Theorem 1.9]{schockQuasilinearTropicalCompactifications2021}, the Kronecker duality map
    \[
        A^1(\oY) \to \Hom(A_1(\oY),\bZ), \;\; D \mapsto (C \mapsto D \cdot C)
    \]
    is an isomorphism. By the same theorem, $A_1(\oY)$ is generated by the 1-dimensional boundary strata of $\oY$. So if
    $D \in A^1(\oY)_{\bQ}$ and $D \cdot C \in \bZ$ for each 1-dimensional boundary stratum $C$ of $\oY$, then $D$
    defines an element of $\Hom(A_1(\oY),\bZ)$, which lifts to an element $D' \in A^1(\oY)$ by surjectivity of the
    Kronecker duality map. Injectivity then implies that $D$ is linearly equivalent to $D'$.
\end{proof}

\section{$W(E_6)$-invariant birational geometry of $\oY$} \label{sec:bir_naruki}

The goal of this section is to study the $W(E_6)$-invariant birational geometry of Naruki's compactification $\oY$. In
\cref{sec:eff_Y} we describe the cones of $W(E_6)$-invariant effective divisors and curves, in \cref{sec:sbl_Y} we
describe the stable base locus decomposition of the $W(E_6)$-invariant effective cone of divisors, and in
\cref{sec:models_Y} we describe the $W(E_6)$-invariant birational models of $\oY$ appearing in this stable base locus
decomposition.

\subsection{The cones of $W(E_6)$-invariant effective divisors and curves} \label{sec:eff_Y}

In this subsection we describe the $W(E_6)$-invariant cones of effective divisors and curves of $\oY$, and as a
consequence also describe the nef cone of $\oY$.

\begin{theorem} \label{thm:eff_Y}
    The $W(E_6)$-invariant effective cone of $\oY(E_6)$ is the closed cone spanned by $B_{A_1}$ and $B_{A_2^3}$.
\end{theorem}

\begin{proof}
    Note that $B_{A_1}$ and $B_{A_2^3}$ themselves are effective divisors. Let
    \[
        B = c_{A_1}B_{A_1} + c_{A_2^3}B_{A_2^3}.
    \]
    be a $W(E_6)$-invariant effective divisor on $\oY(E_6)$. (Any such divisor has this form by \cref{prop:int_Y}.) We
    wish to show $c_{A_1},c_{A_2^3} \geq 0$. By subtracting the fixed components of $B$, we can assume that $B$ does not
    contain any irreducible boundary divisor, so that the restriction of $B$ to any irreducible boundary divisor is
    effective. By \cref{prop:A1_naruki}, we have
    \begin{align*}
        B\vert_{D_{A_1}} = \frac{4c_{A_1}B_2 + (5c_{A_2^3}-3c_{A_1})B_3}{5}.
    \end{align*}
    Since $D_{A_1} \cong \oM_{0,6}$ has symmetric effective cone generated by its symmetric boundary divisors
    \cite{keelContractibleExtremalRays1996}, it follows that $c_{A_1},c_{A_2^3} \geq 0$.
\end{proof}

\begin{theorem} \label{thm:mori_Y}
    The $W(E_6)$-invariant cone of curves of $\oY(E_6)$ is the closed cone spanned by the sum $B_{3A_1}$ of the curves
    of type $3A_1$, and the sum $B_{2A_1A_2^3}$ of the curves of type $2A_1A_2^3$.
\end{theorem}

\begin{proof}
    It follows by \cref{thm:eff_Y} and \cite[Corollary 2.3]{keelContractibleExtremalRays1996} that the
    $W(E_6)$-invariant cone of curves of $\oY(E_6)$ is generated by curves contained in the boundary. Since the curves
    of types $3A_1$ and $2A_1A_2^3$ generate the symmetric effective cones of $\oM_{0,6}$ and $(\bP^1)^3$ (cf.
    \cite{keelContractibleExtremalRays1996}), the result follows.
\end{proof}

\begin{corollary} \label{cor:nef_Y}
    The $W(E_6)$-invariant nef cone of $\oY(E_6)$ is spanned by $B_{A_1}+3B_{A_2^3}$ and $B_{A_1}+B_{A_2^3}$.
\end{corollary}

\begin{proof}
    The nef cone is the dual of the cone of curves under the intersection pairing, so this follows by a direct
    computation from \cref{thm:mori_Y,prop:int_nums_Y}.
\end{proof}

\begin{remark} \label{rmk:int_divs}
    The first lattice point of a ray in the effective cone is the smallest $\bQ$-divisor $D$ lying on that ray such that
    $D$ is actually a $\bZ$-divisor. It follows from \cref{lem:Zdiv,prop:int_nums_Y} that the first lattice points on
    the rays through $B_{A_1} + 3B_{A_2^3}$ and $B_{A_1}+B_{A_2^3}$ are $\frac{B_{A_1}+3B_{A_2^3}}{4} = K_{\oY} +
    \frac{1}{2}B$ and $\frac{B_{A_1}+B_{A_2^3}}{2} = \frac{1}{2}B$, where $B=B_{A_1}+B_{A_2^3}$ is the sum of all the
    boundary divisors.
\end{remark}

\subsection{Stable base locus decomposition} \label{sec:sbl_Y}

\begin{theorem} \label{thm:sbl_Y}
    Let $D$ be a $W(E_6)$-invariant effective divisor on $\oY$.
    \begin{enumerate}
        \item If $D \in [B_{A_1}+3B_{A_2^3}, B_{A_1}+B_{A_2^3}]$, then $D$ is nef and
            semi-ample. In particular, $\bB(D) = \emptyset$.
        \item If $D \in (B_{A_1}+3B_{A_2^3},B_{A_2^3}]$, then $\bB(D) = B_{A_2^3}$ is the sum of the $A_2^3$
            divisors.
        \item If $D \in (B_{A_1}+B_{A_2^3},5B_{A_1}+3B_{A_2^3}]$, then $\bB(D)=B_{2A_1}$ is the sum of
            the $2A_1$ surfaces.
        \item If $D \in (5B_{A_1}+3B_{A_2^3},B_{A_1}]$, then $\bB(D)=B_{A_1}$ is the sum of the $A_1$
            divisors.
    \end{enumerate}
\end{theorem}

\begin{proof}
    \begin{enumerate}
        \item Let $\Delta_1 = \frac{1}{2}(B_{A_1}+B_{A_2^3})$ and $\Delta_2 = \frac{1}{2}B_{A_1}$, and for $i=1,2$ let
            $D_i = K_{\oY}+\Delta_i$. Then $D_1 = \frac{B_{A_1}+3B_{A_2^3}}{4}$ and $D_2 = \frac{B_{A_1}+B_{A_2^3}}{4}$ span the
            extremal rays of the $W(E_6)$-invariant nef cone of $\oY$ by \cref{cor:nef_Y}. Thus it suffices to show that
            $D_i$ is semi-ample for $i=1,2$. But $D_i$ is nef, and since it lies in the interior of the
            (pseudo)effective cone, $D_i$ is big \cite[Theorem 2.2.26]{lazarsfeldPositivityAlgebraicGeometry2004}.
            (Alternatively, one may directly compute, using \cite[Theorem 4.13]{colomboChowGroupModuli2004}, that
            $D_1^4=27$ and $D_2^4=12$ are both positive, hence $D_i$ is big by \cite[Theorem
            2.2.16]{lazarsfeldPositivityAlgebraicGeometry2004}.) Thus $2D_i-K-\Delta_i=D_i$ is nef and big. Since
            $(\oY,\Delta_i)$ is klt, it follows by the Basepoint-Free Theorem that $D_i$ is semi-ample. (In
            \cref{thm:models} below we will explicitly describe the morphisms given by $\lvert mD_i \rvert$ for $m \gg
            0$.)
        \item Let $D \in (B_{A_1}+3B_{A_2^3},B_{A_2^3}]$. Since $B_{A_1}+3B_{A_2^3}$ is semi-ample by part (1), we see
            that $\bB(D) \subset B_{A_2^3}$. Write
            \[
                D = \alpha B_{A_1} + (3\alpha + \beta)B_{A_2^3}
            \]
            where $\alpha \geq 0$, $\beta > 0$. If $C_{2A_1A_2^3}$ is a curve of type $2A_1A_2^3$, then by
            \cref{prop:int_nums_Y}, we have
            \begin{align*}
                D \cdot B_{2A_1A_2^3} &= 3\alpha - (3\alpha + \beta) \\
                                      &= -\beta < 0.
            \end{align*}
            Since the curves of type $2A_1A_2^3$ cover $B_{A_2^3}$, it follows that $B_{A_2^3} \subset \bB(D)$, hence
            $\bB(D) = B_{A_2^3}$.
        \item Suppose $D \in (B_{A_1}+B_{A_2^3},5B_{A_1}+3B_{A_2^3}]$. We first show that $B_{2A_1} \subset \bB(D)$. Let
            $S \cong \oM_{0,5} \cong Bl_4\bP^2$ be an irreducible $2A_1$ surface. Recall the correspondence between the
            boundary divisors of $\oM_{0,5}$ and $Bl_4\bP^2$: for $ij \subset [4]$, $D_{ij} = \ell_{ij} = h-e_i-e_j$ is
            the strict transform of the line through the $i$th and $j$th blown up points, and for $i \in [4]$,
            $D_{i5}=e_i$ is the exceptional divisor over the $i$th blown up point. The $3A_1$ curves on $S$ are
            precisely the divisors $D_{ij}=\ell_{ij}$. Viewing $S$ as $\oM_{0,5}$, let $C$ be the moving curve obtained
            by fixing 4 points on $\bP^1$ and varying the 5th point. Then $C \cdot D_{i5} = 1$ and $C \cdot D_{ij} = 0$.
            It follows that the class of $C$ is $2h-e_1-e_2-e_3-e_4 = \ell_{ij} + \ell_{kl}$, for any partition $ij
            \amalg kl = [4]$. By \cref{prop:int_nums_Y}, $D \cdot C_{3A_1} < 0$, hence $D \cdot C <  0$. Since $C$
            covers $S$, we conclude that $B_{2A_1} \subset \bB(D)$.

            Now we show that the stable base locus of $D$ is exactly $B_{2A_1}$. For this, note that since
            $B_{A_1}+B_{A_2^3}$ is semi-ample by part (1), $\bB(D)$ is contained in the stable base locus of the divisor
            $\Delta = \frac{5B_{A_1}+3B_{A_2^3}}{4}$ lying on the other ray of the cone
            $(B_{A_1}+B_{A_2^3},5B_{A_1}+3B_{A_2^3}]$.  Thus it suffices to show that $\bB(\Delta) \subset B_{2A_1}$.

            Since $\Delta \in (B_{A_1}+B_{A_2^3},B_{A_1}]$, and $B_{A_1}+B_{A_2^3}$ is semi-ample, it follows that
            $\bB(\Delta) \subset B_{A_1}$. If $D_{A_1}$ is an $A_1$ divisor on $\oY$, then by \cref{prop:A1_naruki},
            $\Delta\vert_{D_{A_1}}=B_2$ is the sum of the $D_2$ divisors on $D_{A_1} \cong \oM_{0,6}$; these are
            precisely the $2A_1$ surfaces on $\oY$ that are contained in $D_{A_1}$. Thus
            $\Delta\vert_{B_{A_1}}=B_{2A_1}$ is the sum of all the $2A_1$ surfaces on $\oY$. Now, since
            $\Delta\vert_{B_{A_1}} = B_{2A_1}$ is the fixed effective divisor $B_{2A_1}$ on $B_{A_1}$, we have that
            \[
                H^0(B_{A_1},\Delta\vert_{B_{A_1}}) \neq 0,
            \]
            and any nonzero section of $\Delta\vert_{B_{A_1}}$ vanishes exactly along $B_{2A_1}$. Thus the statement
            that $\bB(\Delta) \subset B_{2A_1}$ follows from the claim that the restriction map
            \[
                H^0(\oY,\Delta) \to H^0(B_{A_1},\Delta\vert_{B_{A_1}})
            \]
            is surjective. In order to prove this claim, note that by taking the long exact sequence of cohomology of
            the exact sequence
            \[
                0 \to \cO_{\oY}(\Delta-B_{A_1}) \to \cO_{\oY}(\Delta) \to \cO_{B_{A_1}}(\Delta\vert_{B_{A_1}}) \to 0,
            \]
            it suffices to prove that
            \[
                H^1(\oY,\Delta-B_{A_1}) = 0.
            \]
            For this, observe that
            \[
                \Delta - B_{A_1} = \frac{B_{A_1}+3B_{A_2^3}}{4} = K_{\oY} + \frac{1}{2}B,
            \]
            where $B = B_{A_1}+B_{A_2^3}$ is the sum of the boundary divisors of $\oY$. The $\bQ$-divisor $\frac{1}{2}B$
            is actually a nef and big integral divisor (cf. \cref{cor:nef_Y,rmk:int_divs}), and its support has simple
            normal crossings. Therefore, by Kawamata-Viehweg vanishing,
            \[
                H^1(\oY,\Delta-B_{A_1}) = H^1(\oY,K_{\oY} + \frac{1}{2}B) = 0.
            \]
            (See also \cite[Section 6.1]{colomboChowGroupModuli2004}.)
        \item Let $D \in (5B_{A_1}+3B_{A_2^3},B_{A_1}]$. Note that $D \in (B_{A_1}+B_{A_2^3},B_{A_1}]$, and
            $B_{A_1}+B_{A_2^3}$ is semi-ample by part (1), thus $\bB(D) \subset B_{A_1}$.
            \[
                D = (5\alpha + \beta)B_{A_1} + 3\alpha B_{A_2^3},
            \]
            where $\alpha \geq 0$, $\beta > 0$. By \cref{prop:A1_naruki}, the restriction of $D$ to a given $A_1$
            divisor $\cong \oM_{0,6}$ is
            \[
                \frac{(20\alpha + 4\beta)B_2 + (-3\beta)B_3}{5}.
            \]
            Since $-3\beta < 0$, this is not effective by \cite[Theorem 1.3(1)]{keelContractibleExtremalRays1996}. Thus
            $B_{A_1} \subset \bB(D)$.
    \end{enumerate}
\end{proof}

\subsection{$W(E_6)$-equivariant birational models of $\oY$} \label{sec:models_Y}

\begin{theorem} \label{thm:models}
    Let $D$ be a $W(E_6)$-invariant effective $\bQ$-divisor on $\oY$.
    \begin{enumerate}
        \item If $D \in (B_{A_1}+3B_{A_2^3},B_{A_1}+B_{A_2^3})$, then $D$ is ample, and $\oY(D) = \oY$.
        \item If $D \in [B_{A_1}+3B_{A_2^3},B_{A_2^3})$, then $\oY(D)$ is the GIT moduli space $\oM$ of marked cubic
            surfaces. The morphism $\pi : \oY \to \oM$ is given by contracting the 40 $A_2^3$ divisors of $\oY$ to
            singular points, each locally isomorphic to the cone over the Segre embedding of $(\bP^1)^3$.
        \item If $D$ is a multiple of $B_{A_1}+B_{A_2^3}$, then $D$ is semi-ample and the morphism $\phi : \oY \to
            \oY(D) =: \oW$ associated to $\lvert mD \rvert$ for $m \gg 0$ is given by contracting each $2A_1$ surface $S
            \cong Bl_4\bP^2$ to a singular line via the strict transform of the linear system of conics on $\bP^2$
            passing through the 4 blown up points.
        \item If $D \in (B_{A_1}+B_{A_2^3},B_{A_1})$, then $\oY(D)$ is the $D$-flip $\oX$ of the small contraction $\phi
            : \oY \to \oW$ of part (3).
        \item If $D$ is a multiple of $B_{A_1}$ or $B_{A_2^3}$, then $\oY(D)$ is a point.
    \end{enumerate}
\end{theorem}

\begin{proof}
    \begin{enumerate}
        \item This first part follows from \cref{cor:nef_Y}, since $(B_{A_1}+3B_{A_2^3},B_{A_1}+B_{A_2^3})$ is the
            interior of the nef cone.
        \item The described morphism $\pi : \oY \to \oM$ is explicitly constructed by Naruki in
            \cite{narukiCrossRatioVariety1982}, and Colombo and van Geemen show that \cite[Proposition
            2.7]{colomboChowGroupModuli2004}
            \[
                \pi^*\cO(1) = \cO\left(\frac{B_{A_1}+3B_{A_2^3}}{4}\right).
            \]
            Thus if $D'$ lies on the ray through $B_{A_1}+3B_{A_2^3}$, then $\lvert mD' \rvert$ for $m \gg 0$ gives the
            morphism $\pi$.

            In general, if $D \in (B_{A_1}+3B_{A_2^3},B_{A_2^3})$, then $\lvert mD \rvert$ for $m \gg 0$ has fixed part
            $B_{A_2^3}$ by \cref{thm:sbl_Y}. Therefore $H^0(\oY,mD) \cong H^0(\oY,mD-B_{A_2^3}) \cong H^0(\oY,mD')$, so
            $\oY(D) = \oY(D') = \oM$.
        \item Let $D = \frac{B_{A_1}+B_{A_2^3}}{2} = \frac{1}{2}B$. As shown in (the proof of) \cref{thm:sbl_Y}, $D$ is
            semi-ample, and ample outside of the sum $B_{2A_1}$ of the $2A_1$ surfaces. By \cref{prop:A1_naruki}, the
            restriction of $D$ to an $A_1$ divisor $\cong \oM_{0,6}$ is given by
            \[
                \frac{2B_2+B_3}{5}.
            \]
            It is shown in \cite{moonMoriProgramOverline2015} that a sufficiently large multiple of this divisor gives
            the contraction $\oM_{0,6} \to \cI_4$ to the Igusa quartic, obtained by contracting each irreducible $D_2$
            divisor $\cong Bl_4\bP^2$ via the strict transform of the linear system of conics through the 4 blown up
            points. Since the $D_2$ divisors on $\oM_{0,6}$ are precisely the $2A_1$ surfaces of $\oY$, the result
            follows.
        \item Note that
            \[
                K_{\oY}+\frac{1}{2}B_{A_1} = \frac{B_{A_1}+B_{A_2^3}}{4} \;\; \text{ and } \;\; K_{\oY} + \frac{2}{3}B_{A_1} =
                \frac{5B_{A_1}+3B_{A_2^3}}{12}.
            \]
            Thus if $D \in (B_{A_1}+B_{A_2^3},5B_{A_1}+3B_{A_2^3}]$, then $D$ is a multiple of $K_{\oY} + \alpha B_{A_1}$ for
            some $\alpha \in (1/2,2/3]$. Since $B_{A_1}$ is normal crossings, $(\oY,\Delta)$ is klt. Furthermore, $-D$
            is ample on the fibers of the small contraction $\phi : \oY \to \oW$ (since by the previous part these
            fibers are linear combinations of $3A_1$ curves). Thus $\oY(D)$ is the flip $\oX$ of $\phi$, which exists by
            \cite[Theorem 1.3.3]{cortiFlips3folds4folds2007}.

            If $D \in (5B_{A_1}+3B_{A_2^3},B_{A_1})$, then $\lvert mD \rvert$ for $m \gg 0$ has fixed part $B_{A_1}$ by
            \cref{thm:sbl_Y}. Therefore $H^0(\oY,mD) \cong H^0(\oY,mD-B_{A_1})$. We claim that $mD - B_{A_1}$ is a
            multiple of $5B_{A_1} + 3B_{A_2^3}$; note it follows from the claim and the previous statements that $\oY(D)
            \cong X$.

            To prove the claim, observe that since $D$ is a positive linear combination of $5B_{A_1} + 3B_{A_2^3}$ and
            $B_{A_1}$, we can divide through by the coefficient of $B_{A_1}$ to write $D$ as $(5\alpha + 1)B_{A_1} + 3
            \alpha B_{A_2^3}$, for some positive $\alpha \in \bQ$. Therefore,
            \begin{align*}
                mD - B_{A_1} &= (5m\alpha)B_{A_1} + 3m\alpha B_{A_2^3}.
            \end{align*}
            is a multiple of $5B_{A_1} + 3B_{A_2^3}$.
        \item If $D = B_{A_1}$ or $B_{A_2^3}$, then $D$ is a fixed divisor, so $\oY(D)$ is a point.
    \end{enumerate}
\end{proof}

\begin{remark}
    There are two birational models of $\oM_{0,6}$ appearing in the $S_6$-invariant minimal model program for
    $\oM_{0,6}$---the classically known Segre cubic and Igusa quartic threefolds, see
    \cite{moonMoriProgramOverline2015}. The Segre cubic is obtained by contracting the $D_3$ divisors on $\oM_{0,6}$ to
    singular points, and the Igusa quartic is obtained by contracting the $D_2$ divisors on $\oM_{0,6}$ to singular
    lines. The $W(E_6)$-invariant minimal model program for $\oY(E_6)$ contains the $S_6$-invariant minimal model
    program for $\oM_{0,6}$---the restriction of the morphism $\pi : \oY \to \oM$ of \cref{thm:models}(2) to an $A_1$
    divisor is the contraction of $\oM_{0,6}$ to the Segre cubic, and the restriction of the morphism $\phi : \oY \to
    \oW$ of \cref{thm:models}(3) to an $A_1$ divisor is the contraction of $\oM_{0,6}$ to the Igusa quartic.
\end{remark}

\begin{remark}
    We have been unsuccessful in explicitly constructing the flip $\oX \to \oW$ of the small contraction $\oY \to \oW$.
    Intuitively, since the small contraction $\oY \to \oW$ is given by contracting each $2A_1$ surface $S \cong
    Bl_4\bP^2$ to a line, $S \to \bP^1$, the natural guess for the flip is to blowup $S$ to $S \times \bP^1$, then
    contract it to $\bP^1 \times \bP^1$. However, since the $2A_1$ surfaces do not all intersect transversally, we
    cannot do this for all surfaces simultaneously. First blowing up all intersections of the $2A_1$ surfaces leads to
    the space $\wY$ studied below, and there does not appear to be an evident contraction of $\wY$ to the flip $\oX$
    either.
\end{remark}

\section{The moduli space of stable marked cubic surfaces} \label{sec:wY}

Recall that the moduli space of stable marked cubic surfaces is the blowup $\wY = \wY(E_6)$ of $\oY = \oY(E_6)$ along
all intersections of $A_1$ divisors, in increasing order of dimension \cite[Section 10]{hackingStablePairTropical2009}.
In this section we gather the necessary results concerning divisors and intersection theory on $\wY$, in preparation for
the study of the birational geometry of $\wY$ in the following sections.

\subsection{Divisors on $\wY(E_6)$} \label{sec:div_wY}

There are 6 types of divisors of interest on $\wY(E_6)$, labeled by types $a,b,e,a_2,a_3,a_4$. The divisors of types
$a,b,e$ are the strict transforms of the $A_1$, $A_2^3$, and Eckardt divisors on $\oY$, respectively. The divisors of
type $a_i$, $i=2,3,4$ are the (strict transforms of the) exceptional divisors over the intersections of $i$ $A_1$
divisors in $\oY$. We refer to the divisors of types $a,b,a_2,a_3,a_4$ as the boundary divisors of $\wY$, and the
divisors of type $e$ as the Eckardt divisors of $\wY$. We describe in detail each type of divisor and their
intersections.

\begin{proposition} \label{prop:a_div}
    There are 36 type $a$ divisors on $\wY(E_6)$. A given type $a$ divisor $D_a$ is isomorphic to the blowup of
    $\oM_{0,6}$ at the 15 points $D_{ij} \cap D_{kl} \cap D_{mn}$, and then the strict transforms of the 45 lines
    $D_{ij} \cap D_{kl}$. We write $F_{ij,kl,mn}$ for the strict transform of the exceptional divisor over $D_{ij} \cap
    D_{kl} \cap D_{mn}$, $F_{ij,kl}$ for the exceptional divisor over the strict transform of $D_{ij} \cap D_{kl}$.
    Additionally, we write $F_{ij}$ for the strict transform of the divisor $D_{ij}$ on $\oM_{0,6}$, and $F_{ijk}$ for
    the strict transform of the divisor $D_{ijk}$ on $\oM_{0,6}$.

    The nonempty intersections of $D_a$ with the other divisors are as follows.
    \begin{enumerate}
        \item 15 type $a_4$ divisors, intersecting in the 15 $F_{ij,kl,mn}$'s.
        \item 45 type $a_3$ divisors, intersecting in the 45 $F_{ij,kl}$'s.
        \item 15 type $a_2$ divisors, intersecting in the 15 $F_{ij}$'s.
        \item 10 type $b$ divisors, intersecting in the 10 $F_{ijk}$'s.
        \item 15 Eckardt divisors, intersecting in the strict transforms of the 15 Keel-Vermeire divisors.
    \end{enumerate}
    The class of $D_a\vert_{D_a}$ is
    \begin{align*}
        \frac{-8\sum F_{ij,kl,mn} - 7\sum F_{ij,kl} - 6\sum F_{ij} - 3\sum F_{ijk}}{5}.
    \end{align*}
\end{proposition}

\begin{proof}
    The description of $D_a$ and its intersections with the other boundary divisors follows from \cref{prop:A1_naruki}
    and the construction of $\wY(E_6)$. Recall from that proposition that there are two types of intersections of
    $D_{A_1}$ with an Eckardt divisor. The first type gives a $D_{ij}$ surface on $D_{A_1} \cong \oM_{0,6}$; this
    surface is blown up in the construction of $\wY(E_6)$, separating $D_a$ from the given Eckardt divisor. The second
    type of intersection with an Eckardt divisor gives a Keel-Vermeire divisor on $\oM_{0,6}$, so after the blowups the
    intersection is the strict transform of the Keel-Vermeire divisor.

    We explain the calculation of $D_a\vert_{D_a}$. Recall from \cref{prop:A1_naruki} that in $\oY$, we have
    \[
        D_{A_1}\vert_{D_{A_1}} = \frac{-\sum D_{ij} - 3\sum D_{ijk}}{5}.
    \]
    In the sequence of blowups $\wY \to \oY$, the divisor $D_{A_1}$ either contains each blown up center or is disjoint
    from it.  Standard formulas for the normal bundle of a strict transform \cite[B.6.10]{fultonIntersectionTheory1998}
    therefore imply that
    \begin{align*}
        D_{A_1}\vert_{D_{A_1}} &= \frac{-\sum D_{ij} - 3\sum D_{ijk}}{5} - \sum F_{ij,kl,mn} - \sum F_{ij,kl} - \sum
        F_{ij},
    \end{align*}
    where $D_{ij}$, $D_{ijk}$ denote the classes of the pullbacks of $D_{ij}$ and $D_{ijk}$. Since each $D_{ij} \cap
    D_{kl} \cap D_{mn}$ is contained in 3 $D_{ij}$'s and disjoint from each $D_{ijk}$, and each $D_{ij} \cap D_{kl}$ is
    contained in 2 $D_{ij}$'s and intersects any $D_{ijk}$ either trivially or transversally, it follows that
    \[
        \sum D_{ij} = \sum F_{ij} + 3\sum F_{ij,kl,mn} + 2\sum F_{ij,kl}, \;\; \text{ and } \sum D_{ijk} = \sum F_{ijk}.
    \]
    The result follows.
\end{proof}

\begin{proposition} \label{prop:b_div}
    There are 40 type $b$ divisors on $\wY(E_6)$. A given type $b$ divisor $D_b$ is isomorphic to the blowup of
    $(\bP^1)^3$ at the 27 points $p \times q \times r$ and 27 lines $p \times q \times \bP^1$, $p \times \bP^1 \times
    q$, $\bP^1 \times q \times r$, $p,q,r \in \{0,1,\infty\}$. We write $e_{pqr}$ for the strict transform of the
    exceptional divisor over $p \times q \times r$, and $e_{pqx}$, $e_{pxr}$, $e_{xqr}$ respectively for the exceptional
    divisor over the strict transform of the line $p \times q \times \bP^1$, $p \times \bP^1 \times q$, $\bP^1 \times q
    \times r$.

    The nonempty intersections of $D_b$ with the other boundary divisors are as follows.
    \begin{enumerate}
        \item 27 type $a_3$ divisors, intersecting in the 27 $e_{pqr}$'s.
        \item 27 type $a_2$ divisors, intersecting in the 27 $e_{pqx}$, $e_{pxr}$, and $e_{xqr}$'s.
        \item 9 type $a$ divisors, intersecting in the strict transforms of hypersurfaces $p_1 \times \bP^1 \times
            \bP^1$, $\bP^1 \times p_2 \times \bP^1$, $\bP^1 \times \bP^1 \times p_3$.
        \item 18 Eckardt divisors, intersecting in the strict transforms of the smooth hypersurfaces described in
            \cref{prop:A23_naruki}.
    \end{enumerate}
    The class of $D_b\vert_{D_b}$ is
    \[
        -h_1 - h_2 - h_3.
    \]
\end{proposition}

\begin{proof}
    Again everything follows directly from the blowup construction $\wY \to \oY$ and \cref{prop:A23_naruki}. We note
    here that the class of $D_b\vert_{D_b}$ is the pullback of the class of $D_{A_2^3}\vert_{D_{A_2^3}}$, because the
    divisor $D_{A_2^3}$ is either disjoint from or intersects transversally each blown up center.
\end{proof}

\begin{remark} \label{rmk:Bl9P1xP1}
    We observe for later use that the strict transform of the hypersurface of the form $p_i \times \bP^1 \times \bP^1$,
    etc. on a type $b$ divisor is isomorphic to the blowup $Bl_9(\bP^1 \times \bP^1)$ of $\bP^1 \times \bP^1$ at the 9
    points $p \times q$, $p,q \in \{0,1,\infty\}$. We denote the exceptional divisor over $p \times q$ by $e_{pq}$.
\end{remark}

Recall from \cref{prop:eckardt_naruki} that an Eckardt divisor on $\oY$ is isomorphic to the minimal wonderful
compactification $\oX(D_4)$ of the $D_4$ hyperplane arrangement in $\bP^3$; this is obtained by sequentially blowing up
$\bP^3$ at the 12 points and 16 lines where the hyperplane arrangement $x_i=\pm x_j$ is not normal crossings. Eckardt
divisors are labeled by triples of lines on marked cubic surfaces; we say that two Eckardt divisors have a common line
if the corresponding two triples of lines have a common line.

For completeness we record the following description of the Eckardt divisors on $\wY$, although this will not be used
later in the article.

\begin{proposition} \label{prop:eckardt_div}
    There are 45 Eckardt (type $e$) divisors on $\wY(E_6)$. A given Eckardt divisor $D_e$ on $\wY(E_6)$ is isomorphic to
    the blowup of the corresponding Eckardt divisor $D \cong \oX(D_4)$ on $\oY(E_6)$ along a total of 48 points, then the
    strict transforms of 87 lines, as follows.
    \begin{enumerate}
        \item $36 = 3 \cdot 12$ points: 3 points each on the exceptional divisors over the $A_3$ points (corresponding
            to intersections of 4 $A_1$ divisors on $D$, such as $D_7 \cap D_{56} \cap D_{34} \cap D_{12}$.)
        \item $12$ points:
            \begin{enumerate}
                \item 6 points $(1:1:0:0),\ldots,(0:0:1:1)$, corresponding to intersections of 3 $A_1$ divisors such as
            $D_{12} \cap D_{34} \cap D_{126}$, and
                \item 6 points $(1:-1:0:0),\ldots,(0:0:1:-1)$, corresponding to intersections of 3 $A_1$ divisors such
                    as $D_{126} \cap D_{346} \cap D_{12}$.
            \end{enumerate}
        \item $72 = 12 \cdot 6$ lines, given by the intersection of a $D_4$ hyperplane with an exceptional divisor over
            an $A_3$ point (corresponding to intersections of 3 $A_1$ divisors such as $D_7 \cap D_{56} \cap D_{12}$.)
        \item $15$ lines:
            \begin{enumerate}
                \item 3 lines like $x_i=x_j, x_k=x_l$, corresponding to intersections of 2 $A_1$ divisors such as $D_{12}
                    \cap D_{34}$,
                \item 6 lines like $x_i=x_j, x_i=-x_j$, corresponding to intersections of 2 $A_1$ divisors on $D$ such
                    as $D_{12} \cap D_{126}$,
                \item 3 lines $x_i=x_j, x_k=-x_l$, corresponding to intersections of 2 $A_1$ divisors on $D$ such as
                    $D_{12} \cap D_{346}$, and
                \item 3 lines $x_i=-x_j,x_k=-x_l$ (corresponding to intersections of 2 $A_1$ divisors on $D$ such as
                    $D_{126} \cap D_{346}$).
            \end{enumerate}
    \end{enumerate}
    The nonempty intersections of $D_e$ with other divisors are as follows.
    \begin{enumerate}
        \item 36 $a_4$ divisors, intersecting in the 36 exceptional divisors in (1) above.
        \item $12+72=84$ $a_3$ divisors, intersecting in the $12+72$ exceptional divisors in (2) and (3) above.
        \item $27$ $a_2$ divisors.
            \begin{enumerate}
                \item 15 $a_2$ divisors, intersecting in the 15 exceptional divisors in (4) above.
                \item 12 $a_2$ divisors, intersecting in the strict transforms of the exceptional divisors over the
                    original 12 $A_3$ points.
            \end{enumerate}
        \item 12 $a$ divisors, intersecting in the strict transforms of the $D_4$ hyperplanes.
        \item 16 $b$ divisors, intersecting in the strict transforms of the $A_2$ lines.
        \item Every Eckardt divisor:
            \begin{enumerate}
                \item If $D_e$ and $D_e'$ do not have a common line, then there is a unique third Eckardt divisor
                    $D_e''$ such that the intersection of any two of $D_e,D_e',D_e''$ is the same as the intersection of
                    all three. This is the strict transform of one of the 16 quadrics on $\oX(D_4)$, as in
                    \cref{prop:eckardt_naruki}.  (There are 16 of these pairs.)
                \item If $D_e$ and $D_e'$ have a common line, then their intersection is the strict transform of an
                    $F_4$ hyperplane $x_i=0$ or $x_1\pm x_2 \pm x_3 \pm x_4=0$. (There are 12 of these.)
            \end{enumerate}
    \end{enumerate}
\end{proposition}

\begin{proof}
    This again follows from the blowup construction $\wY \to \oY$ and \cref{prop:eckardt_naruki}, as can be seen by
    considering a specific example, see \cref{ex:eckardt_naruki}. In the proposition we have described the corresponding
    blowups induced on this particular example.
\end{proof}

\begin{notation} \label{not:Bl7P2}
    Let $Bl_7\bP^2$ be the blowup of $\bP^2$ at the four points $p_1=(1:0:0)$, $p_2=(0:1:0)$, $p_3=(0:0:1)$,
    $p_4=(1:1:1)$, followed by the three points $p_5=(1:1:0)$, $p_6=(1:0:1)$, $p_7=(0:1:1)$. Denote the exceptional divisors
    by $e_1,\ldots,e_7$. Let $\ell_{ijk}$ denote the strict transform of the line through the points $p_i$, $p_j$, and
    $p_k$ (so $ijk = 125$, $345$, $136$, $246$, $147$, or $237$), and $\ell_{ij}$ the strict transform of the line
    through the points $p_i$ and $p_j$, $i,j \in \{5,6,7\}$.
\end{notation}

\begin{proposition} \label{prop:a2_div}
    There are $270$ type $a_2$ divisors. A given type $a_2$ divisor $D_{a_2}$ is isomorphic to $Bl_7\bP^2 \times
    \bP^1$, where $Bl_7\bP^2$ is as in \cref{not:Bl7P2}. Let $h_1,h_2$ denote the classes of the pullbacks of general
    hyperplanes from $\bP^2$ and $\bP^1$, and let $e_i$ denote the class of $e_i \times \bP^1$.

    The nonempty intersections of $D_{a_2}$ with the other boundary divisors are as follows.
    \begin{enumerate}
        \item 3 type $a_4$ divisors, intersecting in the 3 divisors $e_i \times \bP^1$ for $i=5,6,7$.
        \item 6 type $a_3$ divisors, intersecting in the 6 divisors $\ell_{ijk} \times \bP^1$, where $\ell_{ijk}$ is as
            in \cref{not:Bl7P2}. The class of such an intersection is $h_1-e_i-e_j-e_k$.
        \item 2 type $a$ divisors, intersecting in 2 distinct divisors $Bl_7\bP^2 \times pt$ of class $h_2$.
        \item 4 type $b$ divisors, intersecting in the 4 divisors $e_i \times \bP^1$, for $i=1,2,3,4$.
        \item 5 Eckardt divisors:
            \begin{enumerate}
                \item 3 intersecting in $\ell_{ij} \times \bP^1$, where $\ell_{ij}$, $i,j \in \{5,6,7\}$ is the strict
                    transform of the line through the $i$th and $j$th blown up points. The class of such an intersection
                    is $h_1 - e_i - e_j$.
                \item 2 intersecting in 2 distinct divisors $Bl_7\bP^2 \times pt$ of class $h_2$.
            \end{enumerate}
    \end{enumerate}
    The class of $D_{a_2}\vert_{D_{a_2}}$ is
    \[
        -7h_1+3(e_1+e_2+e_3+e_4)+e_5+e_6+e_7-h_2.
    \]
\end{proposition}

\begin{proof}
    The divisor $D_{a_2}$ is obtained as the divisor over a $2A_1$ surface $S$ in $\oY(E_6)$. The surface $S$ is
    isomorphic to $\oM_{0,5} \cong Bl_4\bP^2$. Let $D$ be an $A_1$ divisor on $\oY$ containing $S$. By \cite[Section
    4]{colomboChowGroupModuli2004}, $N_{D/\oY}\vert_S = N_{S/D} \cong \cO(-1)$, the pullback to $Bl_4\bP^2$ of
    $\cO_{\bP^2}(-1)$, thus from the standard exact sequence
    \[
        0 \to N_{S/D} \to N_{S/\oY} \to N_{D/\oY}\vert_S \to 0,
    \]
    we see that $N_{S/\oY} \cong \cO(-1)^2$ (see also \cite[Lemma 10.11]{hackingStablePairTropical2009}). Let $S'$ be
    the strict transform of $S$ under the blowup of the $4A_1$ points. Then $S$ contains three of the $4A_1$ points, and
    $S' \cong Bl_7\bP^2$, with normal bundle $N' = \cO(-1)^2 \otimes \cO(-e_5-e_6-e_7)$. Next let $S''$ be the strict
    transform of $S$ under the blowup of the $3A_1$ curves. Then $S''$ is obtained from $S'$ by blowing up the strict
    transforms of the 6 lines $x_i=0$, $x_i=x_j$.  Since these are divisors on $S'$, we see that $S'' \cong S'$, but the
    normal bundle to $S''$ is
    \[
        N'' = N' \otimes \cO(-6h_1 + 3(e_1+e_2+e_3+e_4) + 2(e_5+e_6+e_7)).
    \]
    Thus, blowing up $S''$, one obtains that $D_{a_2} \cong S'' \times \bP^1 \cong Bl_7\bP^2 \times \bP^1$, with
    \begin{align*}
        D_{a_2}\vert_{D_{a_2}} &= -h_1-h_2 - (e_5-e_6-e_7) - 6h_1 + 3(e_1+e_2+e_3+e_4) + 2(e_5+e_6+e_7) \\
                               &= -7h_1 + 3(e_1+e_2+e_3+e_4) + (e_5+e_6+e_7) - h_2,
    \end{align*}
    as desired. The intersections of $D_{a_2}$ with the other divisors are immediately verified. We explain the 2 types
    of intersections with Eckardt divisors. Recall from \cref{prop:A1_naruki} that the $2A_1$ surface $S$ in $\oY$ is a
    divisor of the form $D_{ij}$ on $D \cong \oM_{0,6}$. There are 2 types of intersections of $D$ with Eckardt divisors
    on $\oY$. The first type gives a Keel-Vermeire divisor; there are 3 of these intersection $D_{ij}$, giving the first
    type of intersection of $D_{a_2}$ with an Eckardt divisor. The second type comes in pairs $D_e \cap D = D_e' \cap D
    = D_{ij}$, giving the second type of intersection of $D_{a_2}$ with an Eckardt divisor.
\end{proof}

\begin{proposition} \label{prop:a3_div}
    There are $540$ type $a_3$ divisors. A given type $a_3$ divisor $D_{a_3}$ is isomorphic to $\bP^1 \times Bl_3\bP^2$,
    where $Bl_3\bP^2$ is the blowup of $\bP^2$ at the 3 coordinate points, with exceptional divisors $e_1,e_2,e_3$. Let
    $h_1,h_2$ be the classes of the pullbacks of general hyperplanes from $\bP^1$ and $\bP^2$, and let $e_i$ be the
    class of the divisor $\bP^1 \times e_i$.

    The nonempty intersections of $D_{a_3}$ with the other boundary divisors are as follows.
    \begin{enumerate}
        \item 1 type $a_4$ divisor, intersecting in a divisor $pt \times Bl_3\bP^2$ of class $h_1$.
        \item 3 type $a_2$ divisors, intersecting in the 3 divisors $\bP^1 \times e_i$ of class $e_i$, $i=1,2,3$.
        \item 3 type $a$ divisors, intersecting in the 3 divisors $\bP^1 \times \ell_{ij}$, where $\ell_{ij}$ is the
            strict transform of the coordinate line through the $i$th and $j$th coordinate points. The class of such an
            intersection is $h_2-e_i-e_j$.
        \item 2 type $b$ divisors, intersecting in 2 distinct divisors $pt \times Bl_3\bP^2$ of class $h_1$.
        \item 7 Eckardt divisors:
            \begin{enumerate}
                \item 1 intersecting in a divisor $pt \times Bl_3\bP^2$ of $h_1$.
                \item 3 pairs, where the 2 divisors in a given pair intersect in $\bP^1 \times \ell_i$, $\bP^1 \times
                    \ell_i'$, with $\ell_i,\ell_i'$ the strict transforms of 2 general lines meeting the $i$th blown up
                    point. The class of such an intersection is $h_2-e_i$.
            \end{enumerate}
    \end{enumerate}
    The class of $D_{a_3}\vert_{D_{a_3}}$ is
    \[
        -2h_1-h_2.
    \]
\end{proposition}

\begin{proof}
    In $\oY$, let $C$ be a $3A_1$ curve and $S$ a $2A_1$ surface containing $C$. The curve $C$ in $S \cong Bl_4\bP^2$
    appears as the strict transform of one of the line $\ell_{ij}$ with class $h-e_i-e_j$ through 2 of the blown up
    points. It follows that $N_{C/S} = \cO(-1)$. Since $N_{S/\oY} = \cO(-1)^2$ (see the proof of \cref{prop:a2_div}), it
    follows from the standard exact sequence
    \[
        0 \to N_{C/S} \to N_{C/\oY} \to N_{S/\oY}\vert_{C} \to 0
    \]
    that $N_{C/\oY} = \cO(-1)^3$. Let $C'$ be the strict transform of $C$ under the blowup of all $4A_1$ points. Since
    $C$ contains exactly one $4A_1$ point, we see that $C' \cong C \cong \bP^1$, but with normal bundle $N' \cong
    \cO(-2)^3$. Therefore, blowing up $C'$, one obtains exceptional divisor $F \cong \bP^1 \times \bP^2$ with $F\vert_F
    = -2h_1-h_2$. Observe that $F$ intersects the strict transforms of 3 $2A_1$ surfaces transversally in the 3 lines
    $\bP^1 \times p_i$, where $p_i$, $i=1,2,3$ are the coordinate points on $\bP^2$. Thus, blowing up the $2A_1$
    surfaces, one obtains that the strict transform $D_{a_3}$ of $F$ is isomorphic to $\bP^1 \times Bl_3\bP^2$, with
    $D_{a_3}\vert_{D_{a_3}} = -2h_1-h_2$. The intersections with the other divisors are immediately verified. (As in the
    proof of \cref{prop:a2_div}, the 2 types of intersections with Eckardt divisors are described, respectively, by
    Eckardt divisors which restrict to Keel-Vermeire or $D_{ij}$ divisors on an $A_1$ divisor $D$ containing $S$ in
    $\oY$.)
\end{proof}

\begin{proposition} \label{prop:a4_div}
    There are $135$ type $a_4$ divisors. A given type $a_4$ divisor $D_{a_4}$ is the blowup of $\bP^3$ at the 4
    coordinate points and the 6 lines between them. Let $h$ be the class of the pullback of a general hyperplane, $e_i$
    the class of the strict transform of the exceptional divisor over the $i$th blown up point, and $e_{ij}$ the class
    of the exceptional divisor over the strict transform of the line through the $i$th and $j$th point.

    The nonempty intersections of $D_{a_4}$ with the other boundary divisors and Eckardt divisors are as follows.
    \begin{enumerate}
        \item 4 type $a_3$ divisors, intersecting in the 4 $e_{i}$'s.
        \item 6 type $a_2$ divisors, intersecting in the 6 $e_{ij}$'s.
        \item 4 type $a$ divisors, intersecting in the strict transforms of the 4 coordinate hyperplanes.  The class of
            such an intersection is $h-e_i-e_j-e_k-e_{ij}-e_{ik}-e_{jk}$, where $i,j,k \in [4]$ are distinct.
        \item 12 Eckardt divisors, coming in 6 pairs. The 2 Eckardt divisors in a given pair are the strict transforms
            of 2 general planes containing one of the blown up lines. (Note these planes are disjoint on $D_{a_4}$.) The
            class of such an intersection is $h-e_i-e_j-e_{ij}$.
    \end{enumerate}
    The class of $D_{a_4}\vert_{D_{a_4}}$ is $-h$.
\end{proposition}

\begin{proof}
    Let $p$ be a $4A_1$ point in $\oY$. Blowing up $p$, one obtains the exceptional divisor $F \cong \bP^3$ with
    $F\vert_F = -h$. Observe that $F$ intersects 4 $3A_1$ curves transversally in the 4 coordinate points on $\bP^3$,
    and 6 $2A_1$ surfaces transversally in the 6 lines through 2 of the coordinate points on $\bP^3$. The result
    follows.
\end{proof}

\begin{remark}
    The intersections of boundary divisors and Eckardt divisors on $\wY$ can also be ascertained from the explicit
    descriptions of the surfaces parameterized by boundary divisors, see \cite{schockModuliWeightedStable2023}.
\end{remark}

\subsection{Curve and surface strata}

We label the types of strata in $\wY$ by juxtaposition of the types of the corresponding divisors---thus for instance
$aa_2e$ denotes a curve stratum formed by the intersection of a divisor of type $a$, a divisor of type $a_2$, and a
divisor of type $e$. We say a stratum given by the intersection of divisors $D_1,\ldots,D_k$ is a \emph{boundary
stratum} if each $D_i$ is a boundary divisor (i.e., of type $a$, $b$, $a_2$, $a_3$, or $a_4$), and a \emph{mixed
stratum} if at least one $D_i$ is an Eckardt divisor.

\begin{notation}
    Let $\Gamma$ be the collection of 1-dimensional boundary strata as well as the 1-dimensional mixed strata of type
    $aa_2e$. Explicitly, $\Gamma$ is the collection of curve strata of $\wY$ of the following types:
    \begin{align*}
        aa_2a_3, aa_2a_4, aa_3a_4, a_2a_3a_4, aa_2b, aa_3b, a_2a_3b, aa_2e.
    \end{align*}
\end{notation}

We will see later that the curves in $\Gamma$ generate the cone of curves of $\wY$ (\cref{thm:mori_cone}).

\begin{proposition} \label{prop:2dstrata}
    The boundary surface strata as well as the curve strata in $\Gamma$ are described in \cref{tab:curve_surface}, where
    $Bl_7\bP^2$ and $Bl_9(\bP^1 \times \bP^1)$ are as in \cref{rmk:Bl9P1xP1,not:Bl7P2}.
    \begin{table}[htpb]
        \centering
        \caption{Curves on surface strata.}
        \label{tab:curve_surface}
        \begin{tabular}{|c|c|c|c|c|c|c|c|c|c|}
            \hline
            \multicolumn{2}{|c|}{} & \multicolumn{8}{c|}{Curve type} \\
            \hline
            Stratum & Surface  & $aa_2a_3$ & $aa_2a_4$ & $aa_3a_4$ & $a_2a_3a_4$ & $aa_2b$ & $aa_3b$ & $a_2a_3b$ & $aa_2e$ \\
         \hline
             $aa_2$ & $Bl_7\bP^2$ & $\ell_{ijk}$ & $e_5,\ldots,e_7$ & & & $e_1,\ldots,e_4$ & & & $\ell_{ij}$ \\
             $aa_3$ & $\bP^1 \times \bP^1$ & $\bP^1 \times pt$ & & $pt \times \bP^1$ & & & $pt \times \bP^1$ & & \\
             $aa_4$ & $Bl_3\bP^2$ & & $\ell_{ij}$ & $e_i$ & & & & & \\
             $a_2a_3$ & $\bP^1 \times \bP^1$ & $\bP^1 \times pt$ & & & $pt \times \bP^1$ & & & $pt \times \bP^1$ & \\
             $a_2a_4$ & $\bP^1 \times \bP^1$ & & $\bP^1 \times pt$ & & $pt \times \bP^1$ & & & & \\
             $a_3a_4$ & $Bl_3\bP^2$ & & & $\ell_{ij}$ & $e_i$ & & & & \\
             $ab$ & $Bl_9(\bP^1 \times \bP^1)$ & & & & & $pt \times \bP^1$, $\bP^1 \times pt$ & $e_{pq}$ & & \\
             $a_2b$ & $\bP^1 \times \bP^1$ & & & & & $\bP^1 \times pt$ & & $pt \times \bP^1$ & \\
             $a_3b$ & $Bl_3\bP^2$ & & & & & & $\ell_{ij}$  & $e_i$ & \\
            \hline
        \end{tabular}
    \end{table}
\end{proposition}

\begin{proof}
    This is a direct observation given the descriptions of the boundary and Eckardt divisors and their intersections in
    \cref{sec:div_wY}.
\end{proof}

\begin{remark}
    Of course, there are more types of mixed strata than just the curve stratum $aa_2e$---we do not consider them as
    they will not be necessary for our purposes.
\end{remark}

\subsection{Intersection-theoretic results on $\wY(E_6)$}

\begin{proposition} \label{prop:int_formulas_wY}
    \begin{enumerate}
        \item The $W(E_6)$-invariant Picard group of $\wY(E_6)$ is generated by the classes $B_a$, $B_b$, $B_{a_2}$,
            $B_{a_3}$, $B_{a_4}$ of the sums of the divisors of the respective types.
        \item The class $B_e$ of the sum of the Eckardt divisors on $\wY(E_6)$ is given by
            \[
                B_e = \frac{25B_a + 42B_{a_2} + 51B_{a_3} + 52B_{a_4} + 27B_b}{4}
            \]
        \item The canonical class of $\wY$ is given by
            \[
                K_{\wY} = \frac{-B_a + 2B_{a_2} + 5B_{a_3} + 8B_{a_4} + B_b}{4}.
            \]
    \end{enumerate}
\end{proposition}

\begin{proof}
    \begin{enumerate}
        \item Since $\wY$ is obtained from $\oY$ by a sequence of smooth blowups, this follows from \cref{prop:int_Y}
            and the usual formula for the Picard group of a smooth blowup \cite[Exercise
            II.8.5]{hartshorneAlgebraicGeometry1977}. (See also \cite[Theorem
            1.9]{schockQuasilinearTropicalCompactifications2021}.)
        \item In $\oY$, a given $4A_1$ point is contained in four $A_1$ divisors and 12 Eckardt divisors and disjoint from
            all $A_2^3$ divisors, a given $3A_1$ curve is contained in 3 $A_1$ divisors and 6 Eckardt divisors, and is
            either disjoint from or intersects transversally the remaining Eckardt and the $A_2^3$ divisors, and a given
            $2A_1$ curve is contained in 2 $A_1$ divisors and 2 Eckardt divisors, and is either disjoint from or
            intersects transversally the remaining Eckardt and the $A_2^3$ divisors. It follows that
            \begin{align*}
                B_{a} &= B_{A_1} - 2B_{a_2} - 3B_{a_3} - 4B_{a_4}, \\
                B_{b} &= B_{A_2^3}, \\
                B_e &= E - 2B_{a_2} - 6B_{a_3} - 12B_{a_4},
            \end{align*}
            where $B_{A_1}$, $B_{A_2^3}$, and $E$ denote the pullbacks to $\wY$ of the sums of the $A_1$, $A_2^3$, and
            Eckardt divisors on $\oY$, respectively. The result follows using the formula for $E$ from
            \cref{prop:int_Y}.
        \item This is a direct computation using the formulas for $B_{A_1}$, $B_{A_2^3}$ above, the usual formula for
            the canonical class of a smooth blowup \cite[Exercise II.8.5]{hartshorneAlgebraicGeometry1977}, and the
            formula for $K_{\oY}$ from \cref{prop:int_Y}.
    \end{enumerate}
\end{proof}

\begin{proposition} \label{prop:int_nums}
    Intersection numbers of symmetric divisors on $\wY(E_6)$ and curves in $\Gamma$ are given in the following table.
    \begin{table}[htpb]
        \centering
        \caption{Intersection numbers of symmetric divisors on $\wY(E_6)$ and curves in $\Gamma$.}
        \label{tab:int_nums}
        \begin{tabular}{| c || c | c | c | c | c | c | c | c |}
            \hline
            & $aa_2a_3$ & $aa_2a_4$ & $aa_3a_4$ & $a_2a_3a_4$ & $aa_2b$ & $aa_3b$ & $a_2a_3b$ & $aa_2e$ \\
            \hline
            \hline
            $B_a$ & $0$ & $0$ & $-1$ & $2$ & $0$ & $-1$ & $2$ & $0$ \\
            $B_{a_2}$ & $0$ & $-1$ & $2$ & $-1$ & $-3$ & $2$ & $-1$ & $-5$ \\
            $B_{a_3}$ & $-2$ & $2$ & $-1$ & $0$ & $3$ & $-1$ & $0$ & $2$ \\
            $B_{a_4}$ & $1$ & $-1$ & $0$ & $0$ & $0$ & $0$ & $0$ & $2$ \\
            $B_{b}$ & $2$ & $0$ & $0$ & $0$ & $-1$ & $0$ & $0$ & $0$ \\
            $B_{e}$ & $1$ & $2$ & $2$ & $2$ & $0$ & $2$ & $2$ & $-1$ \\
            \hline
        \end{tabular}
    \end{table}
\end{proposition}

\begin{proof}
    This is a direct calculation analogous to the proof of \cref{prop:int_nums_Y}, as follows.

    Fix an $a_2$ divisor $D_{a_2} \cong Bl_7\bP^2 \times \bP^1$, with notation as in \cref{prop:a2_div}. The restrictions of
    the symmetric boundary and Eckardt divisors to $D_{a_2}$ are computed by \cref{prop:a2_div} to be as follows.
    \begin{align*}
        B_a\vert_{D_{a_2}} &= 2h_2, \\
        B_{a_2}\vert_{D_{a_2}} &= -7h_1 + 3(e_1+e_2+e_3+e_4) + (e_5+e_6+e_7) - h_2, \\
        B_{a_3}\vert_{D_{a_2}} &= 6h_1-3(e_1+e_2+e_3+e_4) - 2(e_5+e_6+e_7), \\
        B_{a_4}\vert_{D_{a_2}} &= e_5+e_6+e_7, \\
        B_b\vert_{D_{a_2}} &= e_1+e_2+e_3+e_4, \\
        B_e\vert_{D_{a_2}} &= 3h_1-2(e_5+e_6+e_7) + 2h_2.
    \end{align*}
    Curves of each of the types $aa_2a_3$, $aa_2a_4$, $a_2a_3a_4$, $aa_2b$, $a_2a_3b$, and $aa_2e$ appears on $D_{a_2}$
    as described explicitly in \cref{prop:2dstrata,prop:a2_div}. The intersection numbers with curves of these types
    therefore follows from standard intersection-theoretic computations on $D \cong Bl_7 \bP^2 \times \bP^2$.

    Now let $D_{a_3}$ be an $a_3$ divisor, isomorphic to $\bP^1 \times Bl_3\bP^2$, with notation as in \cref{prop:a3_div}. The
    restrictions of the symmetric boundary and Eckardt divisors to $D$ are computed by \cref{prop:a3_div} to be as
    follows.
    \begin{align*}
        B_a\vert_{D_{a_3}} &= 3h_2 - 2(e_1 + e_2 + e_3), \\
        B_{a_2}\vert_{D_{a_3}} &= e_1 + e_2 + e_3, \\
        B_{a_3}\vert_{D_{a_3}} &= -2h_1 - h_2, \\
        B_{a_4}\vert_{D_{a_3}} &= h_1, \\
        B_b\vert_{D_{a_3}} &= 2h_1, \\
        B_e\vert_{D_{a_3}} &= h_1 + 6h_2 - 2(e_1+e_2+e_3).
    \end{align*}
    Curves of each of the types $aa_2a_3$, $aa_3a_4$, $a_2a_3a_4$, $aa_3b$, $a_2a_3b$ appear on the $a_3$ divisor $D$ as
    described explicitly in \cref{prop:2dstrata,prop:a3_div}. The intersection numbers with curves of these types can
    therefore be found by standard intersection-theoretic computations on $D \cong \bP^1 \times Bl_3\bP^2$. Observe that
    the calculations of these intersection numbers for $aa_2a_3$, $a_2a_3a_4$, and $a_2a_3b$ agree with the analogous
    calculations from restricting to an $a_2$ divisor as above, and the remaining intersection numbers (with the curves
    $aa_3a_4$, $aa_3b$) cover the remainder of \cref{tab:int_nums}.
\end{proof}

\begin{remark} \label{rmk:equiv_curves}
    Observe from \cref{tab:int_nums} that the curves of types $aa_3a_4$ and $aa_3b$ have the same intersection numbers
    with any $W(E_6)$-invariant boundary divisor; likewise for the curves of types $a_2a_3a_4$ and $a_2a_3b$ (cf.
    \cref{tab:curve_surface}). It follows that $B_{aa_3a_4} = B_{aa_3b}$ and $B_{a_2a_3a_4} = B_{a_2a_3b}$ in
    $N_1(\wY)^{W(E_6)}$.
\end{remark}

\begin{remark} \label{rmk:aa2e}
    Similar to the above remark, using \cref{tab:int_nums} one can verify that
    \[
        B_{aa_2e} = B_{aa_2a_3} - B_{aa_2a_4} + 2B_{aa_2b}
    \]
    in $N_1(\wY)^{(W(E_6))}$. This can also be verified at the level of rational equivalence using
    \cref{tab:curve_surface}.  (Note that rational and numerical equivalence coincide on $\wY$ by \cite[Theorem
    1.9]{schockQuasilinearTropicalCompactifications2021}.)
\end{remark}

\section{$W(E_6)$-invariant birational geometry of $\wY(E_6)$} \label{sec:bir_wY}

In this section we study the $W(E_6)$-invariant birational geometry of $\wY=\wY(E_6)$, culminating in a description of
the $W(E_6)$-invariant cones of effective divisors and curves (\cref{thm:eff_cone,thm:mori_cone}), and a complete
description of the log minimal model program for $\wY$ with respect to the divisor $K_{\wY} + cB + dE$, where $B$ is the
sum of the boundary divisors and $E$ is the sum of the Eckardt divisors (\cref{thm:log_mmp}).


In what follows we continue with the notation of \cref{sec:div_wY}.

\subsection{Two-dimensional boundary strata} \label{sec:surf_effs}

\begin{lemma} \label{lem:Bl7eff}
    Let $X=Bl_7\bP^2$ be as in \cref{not:Bl7P2}. The effective cone of $X$ is generated by the $e_i$, $\ell_{ijk}$, and
    $\ell_{ij}$. In particular, a divisor of the form
    \[
        \Delta = c_1\sum_{i=1}^4 e_i + c_2\sum_{i=5}^7 e_i + c_3\sum \ell_{ijk} + c_4\sum \ell_{ij}
    \]
    on $X$ is effective if and only if $c_1,\ldots,c_4 \geq 0$.
\end{lemma}

\begin{proof}
    See, for instance, \cite[Section 7]{castravetCoxRingOverline2009}.
\end{proof}

\begin{lemma} \label{lem:Bl9P1xP1eff}
    Let $X$ be the blowup of $\bP^1 \times \bP^1$ at the 9 points $p \times q$ for $p,q \in \{0,1,\infty\}$, as in
    \cref{rmk:Bl9P1xP1}. Let $\sum e_{pq}$ be the sum of the exceptional divisors, $\sum h_p$ the sum of the strict
    transforms of the rulings $p \times \bP^1$ for $p \in \{0,1,\infty\}$, and $\sum k_q$ the sum of the strict
    transforms of the rulings $\bP^1 \times q$ for $q \in \{0,1,\infty\}$. A divisor on $X$ of the form
    \[
        \Delta = c_1\sum e_{pq} + c_2\sum h_p + c_3\sum k_q
    \]
    is effective if and only if $c_1,c_2,c_3 \geq 0$.
\end{lemma}

\begin{proof}
    By subtracting the fixed components, we can assume that $\Delta$ does not contain any of the $e_{pq},h_p,k_q$. Then
    \begin{align*}
        \Delta \cdot e_{pq} &= -c_1 + c_2 + c_3 \geq 0, \\
        \Delta \cdot h_p &= 3c_1 - 3c_2 \geq 0, \\
        \Delta \cdot k_q &= 3c_1 - 3c_3 \geq 0.
    \end{align*}
    The first inequality gives $c_2+c_3 \geq c_1$, and the latter 2 inequalities give $c_1 \geq c_2,c_3$. So we see that
    $c_2+c_3 \geq c_1 \geq c_2,c_3$, implying that $c_2 \geq 0$ and $c_3 \geq 0$, hence $c_1 \geq 0$ as well.
\end{proof}

For the sake of completeness we also recall the well-known descriptions of the effective cones of the other surface
strata $\bP^1 \times \bP^1$ and $Bl_3\bP^2$. (For proofs, one may use for instance that these are smooth projective
toric varieties, and apply \cite[Lemma 15.1.8]{coxToricVarieties2011}.)

\begin{lemma} \label{lem:P1xP1eff}
    The effective cone of $\bP^1 \times \bP^1$ is generated by the divisor classes $h_1,h_2$ corresponding to the two rulings. In
    particular, a divisor of the form
    \[
        \Delta = c_1h_1 + c_2h_2
    \]
    is effective if and only if $c_1,c_2 \geq 0$.
\end{lemma}

\begin{lemma} \label{lem:Bl3P2eff}
    The effective cone of $Bl_3\bP^2$ is generated by the 3 exceptional divisors $e_1,e_2,e_3$, and the strict
    transforms $\ell_{12},\ell_{13},\ell_{23}$ of the 3 lines passing through 2 of the blown-up points. In particular, a
    divisor of the form
    \[
        \Delta = c_1\sum e_i + c_2\sum \ell_{ij}
    \]
    is effective if and only if $c_1,c_2 \geq 0$.
\end{lemma}

\subsection{Boundary divisors} \label{sec:bdry_effs}

\begin{lemma}
    Let $X$ be the blowup of $\bP^3$ at 4 points in general position, then the 6 lines between them. The effective cone
    of $X$ is generated by the strict transforms of the exceptional divisors $e_i$ over the points, the exceptional
    divisors $e_{ij}$ over the lines, and the strict transforms $h_{ijk}$ of the hyperplanes passing through 3 of the
    points. In particular, a divisor of the form
    \[
        \Delta = c_1\sum e_i + c_2\sum e_{ij} + c_3\sum h_{ijk}
    \]
    is effective if and only if $c_1,c_2,c_3 \geq 0$.
\end{lemma}

\begin{proof}
    Note $X$ is toric and the described divisors are the torus-invariant boundary divisors. The result is
    standard, see \cite[Lemma 15.1.8]{coxToricVarieties2011}.
\end{proof}

\begin{lemma}
    Let $X = \bP^1 \times Bl_3\bP^2$, where $Bl_3\bP^2$ is the blowup of $\bP^2$ at 3 points in general position. In the
    notation of \cref{prop:a3_div}, the effective cone of $X$ is generated by the classes $h_1$, $e_i, i=1,2,3$,
    $\ell_{ij}, i,j \in \{1,2,3\}$ distinct, of the divisors of the forms $pt \times Bl_3\bP^2$, $\bP^1 \times e_i$, and
    $\bP^1 \times \ell_{ij}$, respectively. In particular, a divisor of the form
    \[
        \Delta = c_1h_1 + c_2\sum e_i + c_3\sum \ell_{ij}
    \]
    is effective if and only if $c_1,c_2,c_3 \geq 0$.
\end{lemma}

\begin{proof}
    Again note $X$ is toric and the described divisors are the torus-invariant boundary divisors, so the
    result follows by \cite[Lemma 15.1.8]{coxToricVarieties2011}.
\end{proof}

\begin{lemma}
    Let $X = Bl_7\bP^2 \times \bP^1$, where $Bl_7\bP^2$ is as in \cref{not:Bl7P2}. Then the effective cone of $X$ is
    generated by the classes $h_2$, $e_i, i=1,\ldots,7$, $\ell_{ijk}$, and $\ell_{ij}$, where the notation is as in
    \cref{prop:a2_div}. In particular, a divisor of the form
    \[
        \Delta = c_1\sum_{i=1}^4 e_i + c_2\sum_{i=5}^7 e_i + c_3\sum \ell_{ijk} + c_4\sum \ell_{ij} + c_5h_2
    \]
    is effective if and only if $c_1,c_2,c_3,c_4,c_5 \geq 0$.
\end{lemma}

\begin{proof}
    Note $\eff(X) = \eff(Bl_7\bP^2) \times \eff(\bP^1)$ (cf., for instance, \cite[Lemma
    3.8]{keelContractibleExtremalRays1996}, using the identification $\bP^1 \cong \oM_{0,4}$), so the result follows by
    \cref{lem:Bl7eff}.
\end{proof}

\begin{lemma} \label{lem:eff_a}
    Let $X$ be the blowup of $\oM_{0,6}$ at the 15 points $D_{ij} \cap D_{kl} \cap D_{mn}$ and the 45 lines $D_{ij} \cap
    D_{kl}$. In the notation of \cref{prop:a_div}, a divisor on $X$ of the form
    \begin{align*}
        \Delta = c_1\sum F_{ij,kl,mn} + c_2 \sum F_{ij,kl} + c_3\sum F_{ij} + c_4\sum F_{ijk}
    \end{align*}
    is effective if and only if $c_1,c_2,c_3,c_4 \geq 0$.
\end{lemma}

\begin{proof}
    We can assume without loss of generality that $\Delta$ does not contain any of the $F$'s, so $\Delta\vert_F$ is effective for
    every $F$. The divisor $F_{ij}$ is isomorphic to $Bl_7\bP^2$, and from the blowup construction of $X$ we see that
    \begin{align*}
        F_{ij}\vert_{F_{ij}} &= -h-e_5-e_6-e_7 - (6h-3\sum^4 e_i - 2\sum^7e_i) \\
                             &= \frac{-7\sum \ell_{ijk} - 3\sum_{i=1}^4 e_i - 8\sum_{i=5}^7 e_i}{6}.
    \end{align*}
    Thus we compute that
    \begin{align*}
        \Delta\vert_{F_{ij}} &= \frac{(6c_2-7c_3)\sum \ell_{ijk} + (6c_4-3c_3)\sum_{i=1}^4 e_i + (6c_1-8c_3)\sum_{i=5}^7e_i}{6}.
    \end{align*}
    It follows from \cref{lem:Bl7eff}, that $\Delta\vert_{F_{ij}}$ is effective $\iff 6c_2 \geq 7c_3$, $6c_4 \geq 3c_3$,
    $6c_1 \geq 8c_3$. Thus to show that $c_1,\ldots,c_4 \geq 0$, it suffices to show that $c_3 \geq 0$.  For this, we
    interpret $X$ as the boundary divisor of $\wY(E_6)$ whose general point parameterizes the stable replacement of a
    marked cubic surface with a single $A_1$ singularity obtained by blowing up 6 points on a conic. Let $C$ be the
    curve in $X$ obtained by fixing the first 5 points in general position and varying the 6th point along the conic
    through the first 5. Then $C$ is a moving curve in $X$, so $\Delta \cdot C \geq 0$. On the other hand, note that $C$
    intersects 5 of the boundary divisors $F_{ij}$ when the 6th point coincides with one of the first 5, and otherwise
    $C$ does not intersect any of the $F$'s. We conclude that
    \[
        \Delta \cdot C = 5c_3 \geq 0,
    \]
    so $c_3 \geq 0$.
\end{proof}

\begin{lemma} \label{lem:eff_b}
    Let $X$ be the blowup of $(\bP^1)^3$ at the 27 points $p \times q \times r$ and the 27 lines $p \times q \times
    \bP^1$, $p \times \bP^1 \times r$, $\bP^1 \times q \times r$, $p,q,r \in \{0,1,\infty\}$, as in \cref{prop:b_div}.
    Let $F_p = \sum e_{pqr}$ denote the sum of the strict transforms of the exceptional divisors over the blown up
    points, $F_{\ell} = \sum e_{pqx} + \sum e_{pxr} + \sum e_{xqr}$ the sum of the exceptional divisors over the strict
    transforms of the lines, and $F_s$ the sum of the strict transforms of the hypersurfaces $p \times \bP^1 \times
    \bP^1$, $\bP^1 \times q \times \bP^1$, $\bP^1 \times \bP^1 \times r$, $p,q,r \in \{0,1,\infty\}$. Then a divisor on
    $X$ of the form
    \[
        \Delta = c_1F_p + c_2F_{\ell} + c_3F_s
    \]
    is effective if and only if $c_1,c_2,c_3 \geq 0$.
\end{lemma}

\begin{proof}
    We can assume without loss of generality that $\Delta$ does not contain any of the given hypersurfaces, so that
    $\Delta$ restricts to an effective divisor on each such hypersurface.

    The strict transform of the exceptional divisor over a point is isomorphic to $Bl_3\bP^2$, and the restriction of
    $\Delta$ to such an exceptional divisor is
    \begin{align*}
        \frac{(3c_3-c_1)\sum \ell_{ij} + (3c_{2}-2c_1)\sum e_i}{3},
    \end{align*}
    hence $3c_3 \geq c_1$ and $3c_{2} \geq 2c_1$. Thus it suffices to show $c_1 \geq 0$.

    The exceptional divisor over a line is isomorphic to $\bP^1 \times \bP^1$, and the restriction of $\Delta$ to such an
    exceptional divisor is
    \begin{align*}
        (3c_1-3c_{2})h_1 + (2c_3-c_{2})h_2,
    \end{align*}
    hence $3c_1 \geq 3c_{2}$ and $2c_3 \geq c_{2}$. Since $3c_{2} \geq 2c_1$, the former inequality gives that
    $c_1 \geq 0$, and then the inequalities from the previous restriction give $c_{2},c_3 \geq 0$ as well.
\end{proof}

\subsection{The $W(E_6)$-invariant cones of effective divisors and curves on $\wY(E_6)$}

\begin{theorem} \label{thm:eff_cone}
    The $W(E_6)$-invariant effective cone of $\wY(E_6)$ is the closed cone spanned by the $W(E_6)$-invariant boundary
    divisors $B_a$, $B_{a_2}$, $B_{a_3}$, $B_{a_4}$, and $B_b$.
\end{theorem}

\begin{proof}
    Let
    \[
        \Delta = c_{a_4}B_{a_4} + c_{a_3}B_{a_3} + c_{a_2}B_{a_2} + c_aB_a + c_bB_b
    \]
    be a $W(E_6)$-invariant effective divisor. We wish to show that all coefficients are nonnegative. We can assume
    without loss of generality that $\Delta$ does not contain any boundary divisor, so that the restriction of $\Delta$ to each
    irreducible boundary divisor is effective. In particular
    \[
        \Delta\vert_{D_a} = \frac{(5c_{a_4}-8c_{a})\sum F_{ij,kl,mn} + (5c_{a_3} - 7c_a)\sum F_{ij,kl} + (5c_{a_2} - 6c_a)\sum F_{ij} + (5c_b - 3c_a)\sum F_{ijk}}{5}
    \]
    is effective. Then by \cref{lem:eff_a} all coefficients of $\Delta\vert_{D_a}$ are nonnegative, thus to show all
    coefficients of $\Delta$ are nonnegative it suffices to show that $c_a \geq 0$. For this, let $C$ be the moving curve in
    $\wY(E_6)$ obtained by fixing 5 points in general position and varying the last point along a general line. This
    line intersects the conic through the first 5 points in 2 points, and the line through any 2 of the first 5 points
    in 1 point. It follows that $C \cdot D_7 = 2$ and $C \cdot D_{ij} = 1$ for $ij \subset [5]$, and otherwise $C$ does
    not intersect any boundary divisor of $\wY$. Since $C$ is moving, we conclude that $\Delta \cdot C = 12c_a \geq 0$, so
    $c_a \geq 0$.
\end{proof}

Recall that $\Gamma$ is the set of curves in $\wY(E_6)$ of the types
\[
    aa_2a_3, aa_2a_4, aa_3a_4, a_2a_3a_4, aa_2b, aa_3b, a_2a_3b, aa_2e.
\]
These are the 1-dimensional boundary strata of $\wY(E_6)$, together with the curves of type $aa_2e$ involving the
intersection with an Eckardt divisor. These latter curves appear on 2-dimensional boundary strata of type $aa_2$,
isomorphic to $Bl_7\bP^2$ as the strict transforms of the lines between 2 of the last 3 blown-up points (cf.
\cref{prop:2dstrata}).

\begin{theorem} \label{thm:mori_cone}
    The $W(E_6)$-invariant cone of curves of $\wY(E_6)$ is the closed cone spanned by the $W(E_6)$-invariant curves of
    the types listed in $\Gamma$.
\end{theorem}

\begin{proof}
    Since by \cref{thm:eff_cone}, the $W(E_6)$-invariant effective cone of $\wY$ is generated by the
    $W(E_6)$-invariant boundary divisors, it follows by \cite[Corollary 2.3]{keelContractibleExtremalRays1996} that the
    $W(E_6)$-invariant cone of curves of $\wY$ is generated by curves contained in the boundary. The restriction of
    a $W(E_6)$-invariant boundary divisor of $\wY$ to an irreducible boundary divisor $D$ of $\wY$ can be written as a
    symmetric divisor on $D$ of the appropriate form given in the corresponding lemma in \cref{sec:bdry_effs}. Thus,
    applying that lemma and \cite[Corollary 2.3]{keelContractibleExtremalRays1996}, we see that the $W(E_6)$-invariant
    effective cone of $\wY$ is generated by curves contained in the 2-dimensional boundary strata of $\wY$. Restricting
    a $W(E_6)$-invariant boundary divisor of $\wY$ further to a 2-dimensional boundary stratum $S$, we get a symmetric
    divisor on $S$ of the appropriate form given in the corresponding lemma in \cref{sec:surf_effs}. Then it follows
    from \cref{prop:2dstrata} and the lemmas in \cref{sec:surf_effs} that the effective cones of the 2-dimensional
    boundary strata are generated by the curves in $\Gamma$, so the result follows.
\end{proof}

\begin{corollary} \label{cor:nef_cone}
    A $W(E_6)$-invariant divisor $\Delta$ on $\wY(E_6)$ is nef if and only if it intersects every curve in $\Gamma$
    nonnegatively. Explicitly, the cone of $W(E_6)$-invariant nef divisors on $\wY(E_6)$ is spanned by the rays through
    the following six divisors.
    \begin{align*}
        5B_{a} + 6B_{a_2} + 7B_{a_3} + 8B_{a_4} + 3B_b, \\
        3B_{a} + 6B_{a_2} + 7B_{a_3} + 8B_{a_4} + 3B_b, \\
        B_{a} + 2B_{a_2} + 3B_{a_3} + 2B_{a_4} + 3B_b, \\
        B_{a} + 2B_{a_2} + 3B_{a_3} + 4B_{a_4} + 3B_b, \\
        B_{a} + 2B_{a_2} + 3B_{a_3} + 4B_{a_4} + B_b, \\
        B_{a} + 2B_{a_2} + 3B_{a_3} + 2B_{a_4} + 2B_b.
    \end{align*}
\end{corollary}

\begin{proof}
    Since the nef cone is the dual of the cone of effective curves, the first sentence follows from
    \cref{thm:mori_cone}.

    For the second sentence, first note as in \cref{rmk:equiv_curves,rmk:aa2e} that $N_1(\wY)^{W(E_6)}$ is freely
    generated by the $W(E_6)$-invariant curve classes
    \[
        B_{aa_2a_3}, B_{aa_2a_4}, B_{aa_3a_4}, B_{a_2a_3a_4}, \text{ and } B_{aa_2b},
    \]
    and applying these remarks to \cref{thm:mori_cone}, we see that the cone of $W(E_6)$-invariant effective curves is
    spanned by the curve classes
    \[
        B_{aa_2a_3}, B_{aa_2a_4}, B_{aa_3a_4}, B_{a_2a_3a_4}, B_{aa_2b}, \text{ and } B_{aa_2e} = B_{aa_2a_3} -
        B_{aa_2a_4} + 2B_{aa_2b}.
    \]
    By \cref{prop:int_nums}, this gives respectively the inequalities
    \begin{align*}
        -2c_{a_3} + c_{a_4} + 2c_b & \geq 0, \\
        -c_{a_2} + 2c_{a_3} - c_{a_4} &\geq 0, \\
        -c_a + 2c_{a_2} - c_{a_3} & \geq 0, \\
        2c_a - c_{a_2} &\geq 0, \\
        -3c_{a_2} + 3c_{a_3} - c_b &\geq 0, \\
        -5c_{a_2} + 2c_{a_3} + 2c_{a_4} &\geq 0.
    \end{align*}
    From these inequalities it is a direct calculation to find that the given rays span the $W(E_6)$-invariant nef cone.
\end{proof}

\subsection{The log minimal model program for $\wY(E_6)$}

Let $B$ be the sum of all boundary divisors and $E$ the sum of all Eckardt divisors on $\wY(E_6)$. In this section we
compute the log canonical models of the pair $(\wY(E_6),cB + dE)$, as one varies the coefficients $c$ and $d$.

\begin{lemma} \label{lem:lc_bounds}
    The pair $(\wY(E_6),cB+dE)$ has log canonical singularities as long as $0 \leq c \leq 1$ and $0 \leq d \leq 2/3$.
\end{lemma}

\begin{proof}
    Observe that $\wY(E_6)$ is smooth and $B+E$ has normal crossings except at the triple intersections of three Eckardt
    divisors which have no lines in common. Blowing up these triple intersections therefore gives a log resolution $\pi
    : Z \to \wY(E_6)$ of $(\wY(E_6),B+E)$. Let $F$ be the exceptional locus of $\pi$, i.e., the sum of the exceptional
    divisors over the triple intersections of Eckardt divisors. Then the class of the strict transform of $E$ is $\wE =
    \pi^*E - 3F$, and the class of the strict transform of $B$ is $\wB = \pi^*B$. Thus
    \begin{align*}
        \pi^*(K_{\wY} + cB + dE) &= \pi^*K_{\wY} + c\wB + d\wE + 3dF,
    \end{align*}
    so
    \begin{align*}
        K_Z &= \pi^*K_{\wY} + F = \pi^*(K_{\wY}+cB+dE) - c\wB - d\wE + (1-3d)F.
    \end{align*}
    It follows that $(\wY(E_6),cB+dE)$ is log canonical as long as $c \leq 1$ and $d \leq 2/3$.
\end{proof}

\begin{remark}
    Note analogously that a smooth weighted marked cubic surface $(S,dB)$ with an Eckardt point is log canonical if and
    only if $d \leq 2/3$, see \cite[Section 7.1]{schockModuliWeightedStable2023}.
\end{remark}

\begin{theorem} \label{thm:log_mmp}
    Fix $0 \leq c \leq 1$ and $0 \leq d \leq 2/3$.
    \begin{enumerate}
        \item If $4c + 25d < 1$, then $K_{\wY} + cB + dE$ is not effective.
        \item If $4c + 25d = 1$, then the log canonical model of $(\wY,cB+dE)$ is a point.
        \item If $2c+12d \leq 1$ and $4c+25d > 1$ then the log canonical model of $(\wY,cB+dE)$ is the GIT moduli space
            $\oM$ of marked cubic surfaces, obtained by contracting the 40 $A_3^2$ divisors $\cong (\bP^1)^3$ on $\oY$
            to singular points.
        \item If $c+4d \leq 1$ and $2c+12d > 1$, then the log canonical model of $(\wY,cB+dE)$ is Naruki's
            compactification $\oY$.
        \item If $c+3d \leq 1$ and $c+4d > 1$, then the log canonical model of $(\wY,cB+dE)$ is the blowup $\oY_1$ of
            $\oY$ along the intersections of 4 $A_1$ divisors.
        \item If $c+2d \leq 1$ and $c+3d > 1$, then the log canonical model of $(\wY,cB+dE)$ is the blowup $\oY_2$ of
            $\oY_1$ along the strict transforms of the intersections of 3 $A_1$ divisors.
        \item If $c+2d > 1$, then $K_{\wY}+cB+dE$ is ample, so $(\wY,cB+dE)$ is already its own log canonical model.
            Recall this is the blowup of $\oY_2$ along the strict transforms of the intersections of 2 $A_1$ divisors.
    \end{enumerate}
\end{theorem}

\begin{proof}
    By \cref{prop:int_formulas_wY}, $K_{\wY} + cB + dE$ is given by
    \begin{align*}
        \frac{(-1+4c+25d)B_{a} + (2+4c+42d)B_{a_2} + (5+4c+51d)B_{a_3} + (8+4c+52d)B_{a_4} + (1+4c+27d)B_b}{4}.
    \end{align*}
    By \cref{thm:eff_cone} this is not effective for $4c+25d < 1$.

    If $4c+25d=1$, then the pushforward of $K_{\wY}+cB+dE$ to $\oM$ is $0$, so if we know that $\oM$ is the log canonical
    model for $2c+12d \leq 1$ and $4c+25d > 1$, then it follows that the log canonical model for $4c+25d=1$ is a point.
    Thus from now on we assume $4c+25d > 1$. Let $\Delta_{c,d} = K_{\wY}+cB+dE$, and let $\pi : \wY \to Z$ be the morphism to
    the proposed log canonical model. Then to show $Z$ is indeed the log canonical model, it suffices to show that
    $\Delta_{c,d} - \pi^*\pi_*(\Delta_{c,d})$ is effective, and $\pi_*(\Delta_{c,d})$ is ample, see
    \cite{simpsonLogCanonicalModels2007, moonLogCanonicalModels2013} for more details and similar arguments.

    For $Z \neq \oM$, the map $\pi : \wY \to Z$ is a sequence of smooth blowups, from which one easily computes
    $\pi^*\pi_*(\Delta_{c,d})$. For $Z = \oM$, the map $\pi : \wY \to Z$ is a composition of the blowup $g : \oY \to
    \oM$ of the 40 singular points, followed by the sequence of smooth blowups $\wY \to \oY$. If $\oB_{A_1}$ denotes the
    image of $B_{A_1} \subset \oY$ in $\oM$, then $B_{A_1} = g^*\oB_{A_1} + 3B_{A_2^3}$ by \cite[Corollary
    5.6(1)]{casalaina-martinNonisomorphicSmoothCompactifications2022}, from which one also computes
    $\pi^*\pi_*(\Delta_{c,d})$ in this case. The resulting formulas are given in \cref{tab:pushpull}.

    \begin{table}[htpb]
        \centering
        \caption{Formulas for $\pi^*\pi_*(\Delta_{c.d})$.}
        \label{tab:pushpull}
        \begin{tabular}{| c | c |}
            \hline
            $Z$ & $\pi^*\pi_*(\Delta_{c,d})$ \\
            \hline\hline
            $\oM$ & $\dfrac{(-1+4c+25d)(B_a+2B_{a_2}+3B_{a_3}+4B_{a_4}+3B_b)}{4}$ \\
            $\oY$ & $\dfrac{(-1+4c+25d)(B_a+2B_{a_2}+3B_{a_3}+4B_{a_4})+(1+4c+27d)B_b}{4}$ \\
            $\oY_1$ & $\dfrac{(-1+4c+25d)(B_a+2B_{a_2}+3B_{a_3})+(8+4c+52d)B_{a_4}+(1+4c+27d)B_b}{4}$ \\
            $\oY_2$ & $\dfrac{(-1+4c+25d)(B_a+2B_{a_2})+(5+4c+51d)B_{a_3}+(8+4c+52d)B_{a_4}+(1+4c+27d)B_b}{4}$ \\
            $\wY$ & $\Delta_{c,d}$ \\
            \hline
        \end{tabular}
    \end{table}

    From \cref{tab:pushpull}, we compute $\Delta_{c,d} - \pi^*\pi_*(\Delta_{c,d})$ to be as in \cref{tab:K_effs}, and we
    see that it is effective in the desired region.
    \begin{table}[htpb]
        \centering
        \caption{Effectivity of $\Delta_{c,d} - \pi^*\pi_*(\Delta_{c,d})$.}
        \label{tab:K_effs}
        \begin{tabular}{| c | c | c |}
            \hline
            $Z$ & $\Delta_{c,d} - \pi^*\pi_*(\Delta_{c,d})$ & Effective region \\
            \hline
            \hline
            $\oM$ & $(1-c-2d)B_{a_2} + (2-2c-6d)B_{a_3} + (3-3c-12d)B_{a_4} + (1-2c-12d)B_b$ & $2c+12d \leq 1$ \\
            $\oY$ & $(1-c-2d)B_{a_2}+(2-2c-6d)B_{a_3}+(3-3c-12d)B_{a_4}$ & $c+4d \leq 1$\\
            $\oY_1$ & $(1-c-2d)B_{a_2} + (2-2c-6d)B_{a_3}$ & $c+3d \leq 1$\\
            $\oY_2$ & $(1-c-2d)B_{a_2}$ & $c+2d \leq 1$\\
            $\wY$ & $0$ & Always \\
            \hline
        \end{tabular}
    \end{table}

    It remains to show that $\pi_*(\Delta_{c,d})$ is ample in the desired region. Using \cref{prop:int_nums} and
    \cref{tab:pushpull}, we compute the intersection numbers of $\pi^*\pi_*\Delta_{c,d}$ and the extremal rays of the
    $W(E_6)$-invariant cone of curves of $\wY$ to be as in \cref{tab:K_ints}. By \cref{cor:nef_cone} and the table, we
    see that for $c,d$ in the desired region, $\pi^*\pi_*\Delta_{c,d}$ is nef and zero precisely on the curves
    contracted by $\pi$. The result follows.
    \begin{table}[htpb]
        \centering
        \caption{Intersection numbers of $\pi^*\pi_*\Delta_{c,d}$ with the extremal rays of the $W(E_6)$-invariant cone of curves
        of $\wY(E_6)$.}
        \label{tab:K_ints}
        \begin{tabular}{|c|c|c|c|c|c|c|c|}
            \hline
            & $aa_2a_3$ & $aa_2a_4$ & $aa_3a_4$, $aa_3b$ & $a_2a_3a_4$, $a_2a_3b$ & $aa_2b$ & $aa_2e$ \\
            \hline\hline
            $\oM$ & $-1+4c+25d$ & $0$ & $0$ & $0$ & $0$ & $-1+4c+25d$ \\
            $\oY$ & $d$ & $0$ & $0$ & $0$ & $-1+2c+12d$ & $-1+4c+25d$ \\
            $\oY_1$ & $4-3c-11d$ & $-3+3c+12d$ & $0$ & $0$ & $-1+2c+12d$ & $5-2c+d$ \\
            $\oY_2$ & $c+d$ & $1-c$ & $-2+2c+6d$ & $0$ & $5-4c-6d$ & $9-6c-11d$ \\
            $\wY$ & $c+d$ & $2d$ & $2d$ & $-1+c+2d$ & $2-c$ & $4-c-d$ \\
            \hline
        \end{tabular}
    \end{table}
\end{proof}

\begin{remark}
    In particular, when $d=0$, one recovers the log minimal model program for $(\oY,cB_{\oY})$: $K_{\oY} + cB_{\oY}$ is
    ample for $c > 1/2$, one performs the contraction $\oY \to \oM$ when $c=1/2$, and when $c=1/4$ the log canonical
    model is a point (cf. \cref{thm:models}).
\end{remark}

\bibliographystyle{amsalpha} \bibliography{cubics}
\end{document}